\documentclass[11pt,a4paper]{amsart}

\usepackage{graphicx}
\usepackage{pstricks}
\usepackage{amsmath}
\usepackage{amsfonts}
\usepackage{amssymb}
\usepackage{amsthm}
\usepackage{url}
\usepackage{mathdots}
\input xy
\xyoption{all}
\usepackage{pst-all}

\newtheorem{theorem}{Theorem}
\newtheorem*{theorem*}{Theorem}

\newtheorem{corollary}[theorem]{Corollary}
\newtheorem{definition}[theorem]{Definition}
\newtheorem{example}[theorem]{Example}
\newtheorem{lemma}[theorem]{Lemma}

\newtheorem{proposition}[theorem]{Proposition}
\newtheorem{remark}[theorem]{Remark}
\newtheorem*{remark*}{Remark}

\begin{document}

\author{J. J. S\'anchez-Gabites}
\title{On the shape of attractors of discrete dynamical systems}
\email{JaimeJ.Sanchez@uam.es}
\address{J. J. S{\'{a}}nchez-Gabites \\ Departamento de An\'alisis econ\'omico (M\'etodos cuantitativos) \\ Facultad de Ciencias Econ\'omicas y Empresariales \\ Universidad Aut\'onoma de Madrid \\ 28049 Madrid \\ Spain}
\thanks{The author is supported by a MEC grant (MTM2009-07030)}
\begin{abstract} Let $M$ be a manifold or (more generally) a locally compact, metrizable ANR. If $K$ is an attractor for a flow in $M$, with basin of attraction $\mathcal{A}(K)$, it is well known that the inclusion $i : K \subseteq \mathcal{A}(K)$ is always a shape equivalence. In this paper we investigate to what extent this generalizes to discrete dynamical systems generated by homeomorphisms, proving that it holds if (and only if) $K$ has polyhedral shape. Then we specialize to the case when $M$ is a manifold of dimension $\leq 3$.
\end{abstract}

\maketitle

\section{Introduction}

Shape theory was invented by Borsuk \cite{borsukshape2} as a modification of homotopy theory especially designed to deal with possibly complicated compact spaces. It replaces the constructions and concepts of homotopy theory by suitable counterparts ``\`a la \v{C}ech'' which overlook the local irregularities of the spaces while still capturing their global features. Shape theory has naturally found interesting applications in the realm of dynamical systems, since very frequently the sets of interest there have a very complicated structure which makes homotopy theory unsuitable for their study.

One of the first authors to use shape theoretical techniques in dynamics was Hastings \cite{hastings1} when considering the following problem. Suppose $\varphi$ is a flow on a manifold $M$ and assume that it has a compact, asymptotically stable attractor $K$ whose basin of attraction we denote $\mathcal{A}(K)$. Although it is reasonably easy to prove the existence of the attractor by finding a trapping region, it is much more difficult to describe the attractor itself or its basin of attraction. In general both may have a very complicated structure, and it becomes interesting to study what topological properties are shared by the attractor and its basin of attraction. The general result that has gradually emerged from the work of Hastings and other authors \cite{sanjurjo3,gunthersegal1,hastings1,kapitanski1,sanjurjo1,sanjurjo4} is essentially the following:

\begin{theorem} \label{teo:flujos} Suppose $\varphi$ is a flow on a manifold (or, more generally, a locally compact, metrizable absolute neighbourhood retract). If $K$ is an attractor for $\varphi$ with basin of attraction $\mathcal{A}(K)$, then the inclusion $i : K \subseteq \mathcal{A}(K)$ is a shape equivalence.
\end{theorem}

Let us remark that $i$ is usually \emph{not} a homotopy equivalence and the use of shape theory is crucial here. For instance, suppose $K \subseteq \mathbb{R}^2$ is a connected but not path connected attractor whose basin of attraction is all of $\mathbb{R}^2$ (it is easy to construct such an example). Then $i$ is a shape equivalence by the theorem above but it is not a homotopy equivalence, because homotopically equivalent spaces have the same number of path connected components.

Our goal in this paper is to obtain a version of Theorem \ref{teo:flujos} for discrete dynamics; that is, for homeomorphisms rather than flows. Since the proof of Theorem \ref{teo:flujos} (which is reviewed in Section \ref{sec:2}) relies heavily on the flow to construct suitable homotopies that provide a shape inverse for $i$, one expects that the result may not generally hold for discrete dynamics. This is indeed the case, as illustrated by the very simple Example \ref{ejem:solenoid} below, and it makes the situation quite more complicated than its continuous counterpart. Several authors have studied the relation between an attractor and its basin of attraction in this discrete setting, concentrating on their connectedness \cite{gobbinosardella1}, their \v{C}ech cohomology \cite{gobbino1} or their shape \cite{moronpaco1}. In this paper we are able to: (i) fully characterize when Theorem \ref{teo:flujos} holds for a discrete dynamical system (Theorem \ref{teo:main}), (ii) prove that it always holds in $2$--manifolds (Theorem \ref{teo:2mfds}), and (iii) obtain an ``almost explicitly'' checkable characterization of when it holds in $\mathbb{R}^3$ (Theorem \ref{teo:3var}).

The paper is organized as follows. In Section \ref{sec:statements} we give the precise statements of our results. Section \ref{sec:2} is intended to illustrate why the continuous and the discrete cases of Theorem \ref{teo:flujos} are so different and what difficulties arise when studying the latter. In Section \ref{sec:main} we prove Theorem \ref{teo:main}. The argument proceeds via the Whitehead theorem, showing that the inclusion of the attractor in its basin of attraction induces isomorphisms between the corresponding shape pro--groups. Part of the proof is purely algebraic, and in Section \ref{sec:algebra} we abstract this argument to obtain a result (Theorem \ref{teo:algebra}) that is central to the study of the cases when $M$ is a low dimensional manifold. These are the subject of Sections \ref{sec:2mfds} and \ref{sec:3mfds}.

The few elementary definitions about dynamics that will be used along the paper are recalled at the beginning of Section \ref{sec:statements}. Unfortunately, it seems impossible to do the same with shape theory in a sufficiently condensed manner, so we assume the reader to have some familiarity with the subject. Appendix \ref{ap:shape}, however, includes suitable bibliographic references and reviews some working definitions that will be used along the paper.

\section{Statement of results} \label{sec:statements}

A \emph{discrete dynamical system} on a space $M$ is a homeomorphism $f : M \longrightarrow M$. An \emph{attractor} for $f$ is a compact set $K$, \emph{invariant} under $f$ (that is, $f(K) = K$) and with the following property: it has a neighbourhood $U$ such that every compact set $C \subseteq U$ is \emph{attracted} by $K$, this meaning that, for every neighbourhood $V$ of $K$, there exists $n_0 \in \mathbb{N}$ such that $f^n(C) \subseteq V$ for every $n \geq n_0$. The maximal neighbourhood $U$ with this property is an open invariant subset of $M$ called the \emph{basin of attraction}, denoted $\mathcal{A}(K)$.

The previous definition is completely standard, although it is sometimes useful to work with an equivalent one which is more operational. Here we follow Katok and Hasselblatt \cite{katok1}. Suppose there is a compact set $P \subseteq M$ such that $fP \subseteq {\rm int}\ P$. We say that $P$ is a \emph{trapping region} for $f$. Then $\{f^nP\}_{n \geq 0}$ is a decreasing sequence of compact subsets of $P$, and their intersection $K$ is a nonempty compact set contained in ${\rm int}\ P$. It is easy to check that $K$ is invariant. Also, if $V$ is any neighbourhood of $K$ then the local compactness of $M$ implies that there exists $n_0 \in \mathbb{N}$ such that $f^nP \subseteq V$ for every $n \geq n_0$, so in particular $K$ attracts every compact set $C \subseteq P$. Thus $K$ is an attractor whose basin of attraction contains $P$. Actually, it is easy to show that its basin of attraction is precisely $\mathcal{A}(K) = \bigcup_{n \leq 0} f^nP$. In the sequel we shall tacitly use these descriptions of an attractor and its basin of attraction.
\medskip

With the above definitions out of the way, let us begin by observing that already when $M = \mathbb{R}^3$ Theorem \ref{teo:flujos} does not hold true, at least without additional hypotheses, for discrete dynamical systems:

\begin{example} \label{ejem:solenoid} There exists a homeomorphism $f : \mathbb{R}^3 \longrightarrow \mathbb{R}^3$ having the dyadic solenoid $K$ as an attractor and such that the inclusion $i : K \subseteq \mathcal{A}(K)$ is not a shape equivalence.
\end{example}
\begin{proof} Let $T \subseteq \mathbb{R}^3$ be a solid torus and $f : \mathbb{R}^3 \longrightarrow \mathbb{R}^3$ be a homeomorphism such that $fT$ lies inside ${\rm int}\ T$ as shown in Figure \ref{fig:fig1}. The set $K := \bigcap_{n \geq 0} f^nT$ is the so-called \emph{dyadic solenoid}. Notice that $T$ is a trapping region for $f$, so $K$ is an attractor and $T$ is contained in its basin of attraction.

\begin{figure}[h!]
\begin{pspicture}(0,0)(9.3,3.4)
	\rput[bl](0,0){\scalebox{0.75}{\includegraphics{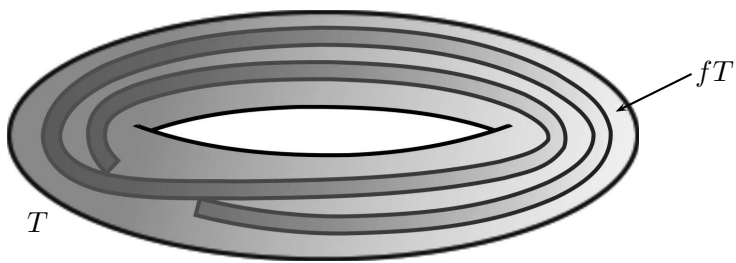}}}
	\psline{->}(9,2.5)(8,2) \rput(9.3,2.5){$fT$}
	\rput[b](0.4,0.4){$T$}
\end{pspicture}
\caption{$T$ and $fT$}
\label{fig:fig1}
\end{figure}

We claim that the inclusion $i$ cannot be a shape equivalence. For, assume on the contrary that it were. Then $i^* : \check{H}^1(\mathcal{A}(K);\mathbb{Z}) \longrightarrow \check{H}^1(K;\mathbb{Z})$ would be an isomorphism. Letting $j : K \subseteq T$ and $k : T \subseteq \mathcal{A}(K)$ denote the inclusions, clearly $i = k j$ so in particular $i^* = j^* k^*$ and we see that $j^* : \check{H}^1(T;\mathbb{Z}) = \mathbb{Z} \longrightarrow \check{H}^1(K;\mathbb{Z})$ would be surjective. But $\check{H}^1(K;\mathbb{Z})$ is not even finitely generated, so this is impossible.
\end{proof}

We have just seen that, in general, we cannot expect Theorem \ref{teo:flujos} to hold in the discrete case. Our first result provides a complete characterization of when it does. Let us call the geometric realization of a locally finite simplicial complex a \emph{polyhedron}\footnote{The simplicial complex itself may be either finite, in which case the polyhedron is compact, or infinite, in which case the polyhedron is noncompact. We warn the reader that some authors reserve the word ``polyhedron'' for compact polyhedra. That is the case in the paper by Kadlof \cite{kadlof1} which we shall use later on.}. Then:

\begin{theorem} \label{teo:main} Let $K$ be an attractor for a homeomorphism of a manifold or, more generally, a locally compact, metrizable absolute neighbourhood retract $M$. The inclusion $i : K \subseteq \mathcal{A}(K)$ is a shape equivalence if, and only if, $K$ has the shape of a polyhedron.
\end{theorem}

It is quite surprising that the shape of $K$ alone suffices to decide whether $i$ is a shape equivalence regardless of the dynamics of $f|_K$. However, assumptions on $f|_K$ can also be profitable:

\begin{corollary} \label{cor:fixed} Assume $f|_K$ is homotopic to the identity ${\rm id}_K$. Then the inclusion $i : K \subseteq \mathcal{A}(K)$ is a shape equivalence.
\end{corollary}

\begin{example} If every point in $K$ is a fixed point for $f$, then $i : K \subseteq \mathcal{A}(K)$ is a shape equivalence.
\end{example}

It is satisfying to observe that Corollary \ref{cor:fixed} allows us to recover Theorem \ref{teo:flujos}:

\begin{example} (Alternative proof of Theorem \ref{teo:flujos}) Suppose $\varphi$ is a flow on $M$ having an attractor $K$ with basin of attraction $\mathcal{A}(K)$. Then the inclusion $i : K \subseteq \mathcal{A}(K)$ is a shape equivalence.
\end{example}
\begin{proof}The time--one map $f : M \longrightarrow M$ defined by $f(p) := \varphi(p,1)$ also has $K$ as an attractor with the same basin of attraction $\mathcal{A}(K)$. Now, the map $H : K \times [0,1] \longrightarrow K$ given by $H(p,t) := \varphi(p,t)$ provides a homotopy between ${\rm id}_K$ and $f|_K$, so Corollary \ref{cor:fixed} applies and $i : K \subseteq \mathcal{A}(K)$ is a shape equivalence.
\end{proof}

Although Theorem \ref{teo:main} is useful for theoretical purposes, the assumption it places on $K$ may be hard to check. This difficulty can be circumvented when $M$ is a low dimensional manifold. The $2$--dimensional case is very neat:

\begin{theorem} \label{teo:2mfds} Let $K$ be a attractor in a $2$--manifold $M$. Then $i : K \subseteq \mathcal{A}(K)$ is a shape equivalence and $K$ has the shape of a compact polyhedron.
\end{theorem}

Suppose now that $K \subseteq \mathbb{R}^3$ is an attractor for a homeomorphism $f$ of $\mathbb{R}^3$. Example \ref{ejem:solenoid} implies that now this situation is more subtle. In Section \ref{sec:3mfds} we shall show that the inclusion $K \subseteq \mathcal{A}(K)$ is a shape equivalence if and only if $K$ has a shape theoretical property called pointed $1$--movability (Theorem \ref{teo:3var1}). This result has an interesting corollary:

\begin{corollary} \label{cor:loccon} Assume that an attractor $K \subseteq \mathbb{R}^3$ has at most countably many path connected components. Then $i : K \subseteq \mathcal{A}(K)$ is a shape equivalence and, furthermore, $K$ has the shape of a compact polyhedron.
\end{corollary}

In practice one usually detects that $f$ has an attractor by finding a trapping region $P$, so that $f$ has an attractor $K \subseteq {\rm int}\ P$ whose basin of attraction contains $P$ and extends beyond it. This approach can be exploited to obtain an almost explicit characterization of when the inclusion $i : K \subseteq \mathcal{A}(K)$ is a shape equivalence as follows. Choose a point $* \in P$ and suppose for our present expository purposes that $P$ is connected and $f$ leaves $*$ fixed. The restriction $f|_P : P \longrightarrow P$ induces a homomorphism $\Phi = (f|_P)_* : \pi_1(P,*) \longrightarrow \pi_1(P,*)$. Consider the decreasing sequence of subgroups of $\pi_1(P,*)$ \[\pi_1(P,*) \geq {\rm im}\ \Phi \geq {\rm im}\ \Phi^2 \geq {\rm im}\ \Phi^3 \geq \ldots\] and say that $\Phi$ \emph{stabilizes} if ${\rm im}\ \Phi^n = {\rm im}\ \Phi^{n+1}$ for some $n$ (one sees easily that then ${\rm im}\ \Phi^n = {\rm im}\ \Phi^{n+1} = {\rm im}\ \Phi^{n+2} = \ldots$). Then the following holds:

\begin{theorem} \label{teo:3var} The inclusion $i : K \subseteq \mathcal{A}(K)$ is a shape equivalence if, and only if, $\Phi$ stabilizes. When this is the case, $K$ has the shape of a compact polyhedron.
\end{theorem}

The condition that $*$ should be a fixed point of $f$ will be removed when we prove Theorem \ref{teo:3var}, at the cost of complicating slightly the definition of $\Phi$. The important point to observe is that both $\pi_1(P,*)$ and $\Phi$ are explicitly computable:
\begin{itemize}
	\item In our present case $M = \mathbb{R}^3$ it is easy to see that $P$ can always be chosen to be a manifold (with boundary). Thus it is possible to obtain a finite presentation of $\pi_1(P,*)$ in terms of generators $g_i$ and relations $r_j$.
	\item Also, direct inspection of how $f$ transforms the loops $g_i$ allows us to write $\Phi(g_i)$ as words in the $g_i$.
\end{itemize}
With these data it is possible to obtain finite systems of generators for ${\rm im}\ \Phi$, ${\rm im}\ \Phi^2$, etc. Then in order to decide whether $\Phi$ stabilizes it is enough to check whether the generators for ${\rm im}\ \Phi^n$ belong to ${\rm im}\ \Phi^{n+1}$ for some $n$, which is a group theoretical problem that in certain cases has a combinatorial solution. In Section \ref{sec:3mfds} we shall exploit an algorithm due to Stallings to deal with the case when $P$ is a handlebody, a common ocurrence in applications.
\medskip

Let us illustrate Theorem \ref{teo:3var} with a very easy example related to Example \ref{ejem:solenoid}:

\begin{example} Assume that the trapping region $P$ is a solid torus. Under the identification $\pi_1(P,*) = \mathbb{Z}$ we can write $\Phi(z) = kz$ for some integer $k \in \mathbb{Z}$. Then $i : K \subseteq \mathcal{A}(K)$ is a shape equivalence if, and only if, $k = 0$ or $k = \pm 1$.
\end{example}
\begin{proof} The sequence ${\rm im}\ \Phi^n$ reads \[\mathbb{Z} \geq {\rm gp}(k) \geq {\rm gp}(k^2) \geq {\rm gp}(k^3) \geq \ldots\] where ${\rm gp}(A)$ means the subgroup of $\mathbb{Z}$ generated by the set $A \subseteq \mathbb{Z}$. According to Theorem \ref{teo:3var} $i$ is a shape equivalence if and only if the equality ${\rm gp}(k^n)={\rm gp}(k^{n+1})$ holds for some $n$, which happens precisely when the generator $k^n$ of the bigger group belongs to the smaller group ${\rm gp}(k^{n+1})$. If $k = 0$ or $k = \pm 1$ this is certainly true already for $n=0$. Conversely, if $k^n \in {\rm gp}(k^{n+1})$ then there exists $z \in \mathbb{Z}$ such that $k^n = k^{n+1}z$ so either $k = 0$ or, if not, cancelling $k^n$ yields $1 = kz$ which implies $k = \pm 1$.
\end{proof}

It will be clear later on that when $k = 0$ then both $K$ and $\mathcal{A}(K)$ have the shape of a point whereas when $k = \pm 1$ they both have the shape of $\mathbb{S}^1$. Also, choosing $P = T$ for the dyadic solenoid of Example \ref{ejem:solenoid} one has $\Phi(z) = 2z$ which shows again that $i$ is not a shape equivalence.

\section{An informal introduction to the proofs} \label{sec:2}

The proof of Theorem \ref{teo:flujos} is very transparent from the geometric point of view because the flow can be used to define suitable homotopies which lead to an explicit construction of the shape inverse for the inclusion $i : K \subseteq \mathcal{A}(K)$. Unfortunately the same is not true of its discrete counterparts: the proofs of Theorems \ref{teo:main}, \ref{teo:2mfds} and \ref{teo:3var} are somewhat involved and mostly algebraic, and the shape inverse for $i$ is not constructed explicitly but just shown to exist via the Whitehead theorem.

In this section we would like to point out the main difference between the continuous and the discrete case and explain why algebraic topology seems unavoidable in handling the latter. For the sake of simplicity we shall not discuss whether $i$ is a shape equivalence but just whether it induces an isomorphism between the one-dimensional homology groups of $K$ and $\mathcal{A}(K)$. This is much weaker than being a shape equivalence but enough for our current heuristic purposes, since it will unearth the main difficulties without any technical burden.

\subsection{The continuous case} Let $\varphi$ be a continuous flow in $\mathbb{R}^2$ that has an attractor $K$ with basin of attraction $\mathcal{A}(K)$. It is a standard result in dynamics that $K$ has a compact neighbourhood $P \subseteq \mathcal{A}(K)$ that is positively invariant (that is, $fP \subseteq P$) and such that the trajectories of $\varphi$ cross the boundary of $P$ transversally\footnote{For instance, one may take $P := L^{-1}[0,1]$, where $L$ is any uniformly unbounded (that is, proper) Lyapunov function for $K$ \cite[Theorem 2.13, p. 73]{bhatiaszego1}.}, if we may borrow this terminology from differential geometry and use it loosely. Let $f : \mathbb{R}^2 \longrightarrow \mathbb{R}^2$ be the time-one map of the flow, so that $f(p) := \varphi(p,1)$. The key observation is that \[(*) \text{ the inclusion } fP \subseteq P \text{ is a homotopy equivalence}.\] The reason is that there exists a strong deformation retraction of $P$ onto $fP$ which is provided by the flow just by pushing every point in $P \backslash fP$ forwards along its trajectory until it first hits $fP$ (the fact that the trajectories of $\varphi$ cross the boundary of $P$ and also $fP$ transversally makes this a continuous map). Property (*) can immediately be generalized to \[(**) \text{ the inclusion } f^{n+1}P \subseteq f^nP \text{ is a homotopy equivalence}.\] Indeed: $f$ takes the pair $(P,fP)$ homeomorphically onto $(fP,f^2P)$, which means geometrically that $fP$ lies inside $P$ in the same way $f^2P$ lies inside $fP$, and so because of (*) the inclusion $f^2P \subseteq fP$ is also a homotopy equivalence. The same is obviously true of each inclusion $f^{n+1}P \subseteq f^nP$.

It is now easy to relate the homology of $K$ and $\mathcal{A}(K)$. Since $K$ is the intersection of the decreasing sequence of sets $P \supseteq fP \supseteq f^2P \supseteq \ldots$ because $P$ is a compact subset of $\mathcal{A}(K)$ and is therefore attracted by $K$, its \v{C}ech homology $\check{H}_1(K)$ is (by definition) the inverse limit of the sequence \[H_1(P) \longleftarrow H_1(fP) \longleftarrow H_1(f^2P) \longleftarrow \ldots\] where each arrow is induced by inclusion. But (**) implies that each arrow is an isomorphism, and so $\check{H}_1(K) \cong H_1(P)$. Writing $\mathcal{A}(K)$ as the union of the increasing sequence of sets $P \subseteq f^{-1}P \subseteq f^{-2}P \subseteq \ldots$ its homology is the direct limit of the sequence \[H_1(P) \longrightarrow H_1(f^{-1}P) \longrightarrow H_1(f^{-2}P) \longrightarrow \ldots\] so the same argument as before shows that $H_1(\mathcal{A}(K)) \cong H_1(P)$. Thus $K$, $P$ and $\mathcal{A}(K)$ all have isomorphic homology.

\subsection{The discrete case: what goes wrong?} Let us turn now to discrete dynamics. Recall that it was property (*) which allowed us to prove, in the continuous case, that the inclusion $K \subseteq \mathcal{A}(K)$ induces an isomorphism in homology. Whereas in the continuous case this property is an automatic consequence of the existence of the flow, in the discrete case it needs to be established \emph{ad hoc} and in fact it may very well happen that (*) does not hold. We shall illustrate this with a couple of examples related to an attractor introduced by Plykin \cite[p. 239]{plykin1} when considering a question of Smale. Following his original example a number of variants have been constructed, and we are going to build on a particular design taken from Kuznetsov \cite[p. 123]{kuznetsov1}.
\smallskip

(A) As a first example, consider Figure \ref{fig:plykin1}. Its left panel shows a disk $P$ with three holes whose boundary curves are labeled $\alpha_i$ for later reference. In panel ({\it b}\/) the disk $P$ is decomposed as the union of three closed regions $A_i$ shown in different shades of gray for pictorial convenience. By first compressing $\mathbb{R}^2$ in the $y$--direction and then stretching in the $x$--direction and bending conveniently it is easy to see that there exists a homeomorphism $f$ of the whole plane that takes each of regions $A_i$ onto the region $fA_i$ of the same colour in the fashion depicted in Figure \ref{fig:plykin2}.({\it a}\/). Notice that the holes of the disk are likewise stretched, so that $fP \subseteq {\rm int}\ P$. Therefore $P$ is a trapping region for $f$ and there exists an attractor $K$ in the interior of $P$; specifically, $K = \bigcap_{n \geq 0} f^nP$ and its basin of attraction is $\mathcal{A}(K) = \bigcup_{n \leq 0} f^nP$. Of course $P$ is then a compact, positively invariant neighbourhood of $K$ contained in $\mathcal{A}(K)$ and it is attracted by $K$.

\begin{figure}[h]
\begin{pspicture}(0,0)(11.6,4.4)
\rput[bl](0,0){\scalebox{0.6}{\includegraphics{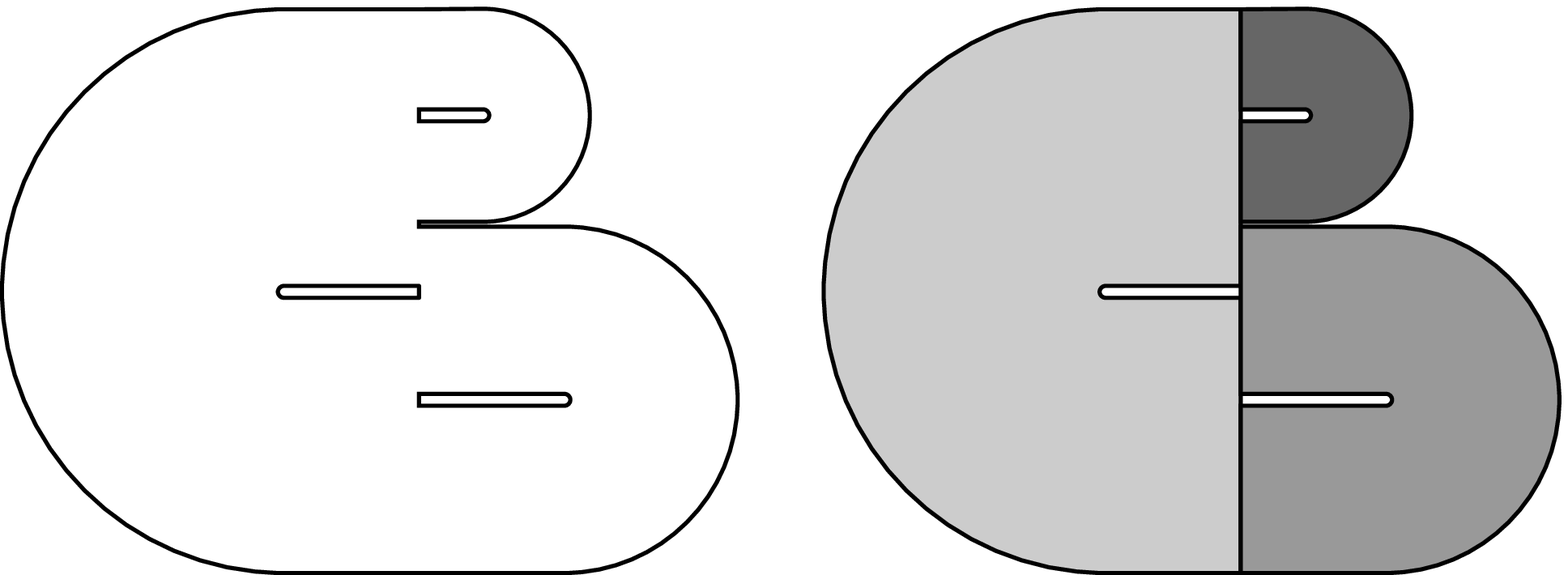}}}
\rput[bl](0.2,4){({\it a}\/)} \rput[bl](6.2,4){({\it b}\/)}
\rput(4.8,2.8){$P$} \rput(2,1.9){$\alpha_1$} \rput(4.4,1.1){$\alpha_2$} \rput(3.8,3.7){$\alpha_3$}
\rput(7.2,1){$A_1$} \rput(10.7,0.5){$A_2$} \rput(10,3.8){$A_3$}
\end{pspicture}
\caption{Setup for a Plykin-type attractor \label{fig:plykin1}}
\end{figure}

\begin{figure}[h]
\begin{pspicture}(0,0)(11.6,4.4)
\rput[bl](0.2,4){({\it a}\/)} \rput[bl](6.2,4){({\it b}\/)}
\rput[bl](0,0){\scalebox{0.6}{\includegraphics{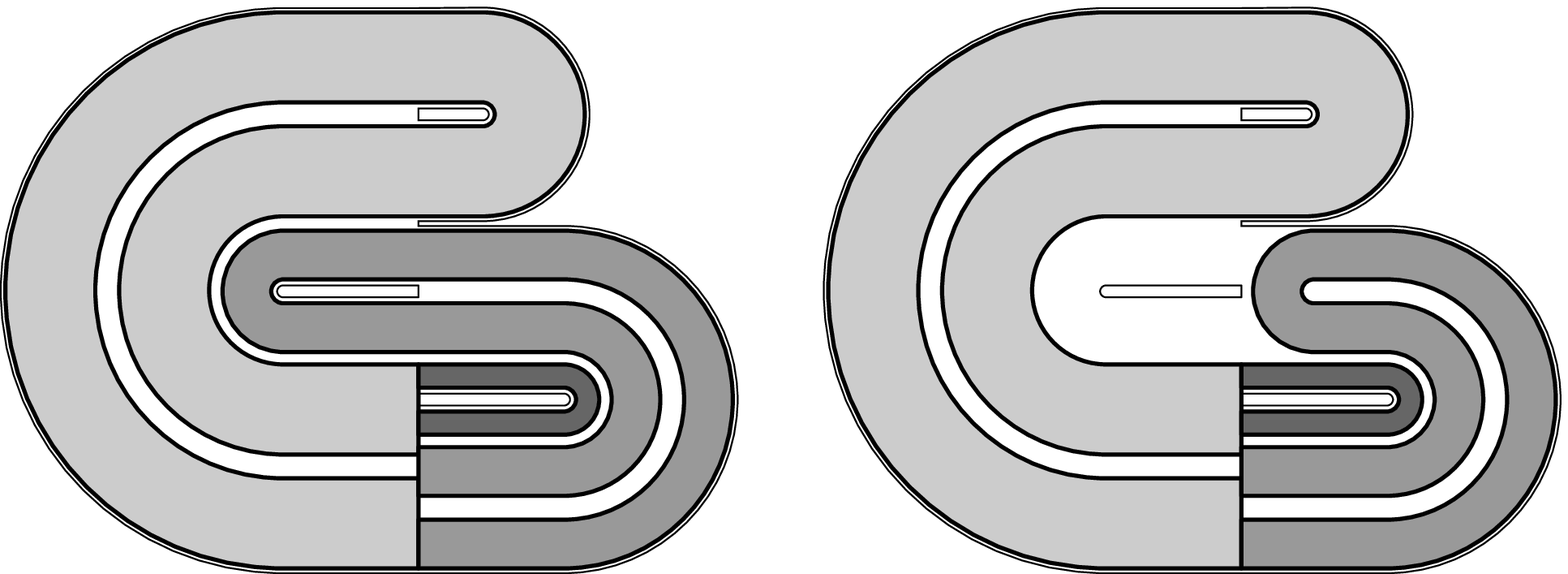}}}
\end{pspicture}
\caption{Setup for a Plykin-type attractor \label{fig:plykin2}}
\end{figure}

Now observe that the inclusion $fP \subseteq P$ is a homotopy equivalence; that is, property (*) holds. In fact, a strong deformation retraction of $P$ onto $fP$ can easily be imagined by looking at Figure \ref{fig:plykin2}.({\it a}\/): simply enlarge each hole of $P$ pushing it outwards until it matches the hole of $fP$ that contains it (also, push the outer rim of $P$ slightly inwards until it matches the outer boundary of $fP$). Thus, property (*) holding true, the same reasoning as in the continuous case proves that the inclusion $K \subseteq \mathcal{A}(K)$ induces an isomorphism in homology. Notice however that, because now the dynamics is discrete and there is no flow, the strong deformation retraction of $P$ onto $fP$ had to be constructed \emph{ad hoc}.
\smallskip

(B) Let us move on to a variant of the above construction, where the map $f$ is defined in such a way that the regions $fA_i$ look as shown in Figure \ref{fig:plykin2}.({\it b}\/). As before $fP \subseteq {\rm int}\ P$ and so there is an attractor $K = \bigcap_{n \geq 0} f^nP$ inside $P$ whose basin of attraction is $\mathcal{A}(K) = \bigcup_{n \leq 0} f^nP$. We are going to see that, even though now property (*) does \emph{not} hold, the inclusion $K \subseteq \mathcal{A}(K)$ still induces isomorphisms in homology. Evidently to show this we need a new, different method of proof.

Since $P$ is a disk with three holes, its first homology group is isomorphic to $\mathbb{Z}^3$ and we may choose as its generators the simple closed curves $\alpha_i$ which bound the holes of $P$ as in Figure \ref{fig:plykin1}.({\it a}\/). Consider how $f$ transforms the loops $\alpha_i$. From Figure \ref{fig:plykin2}.({\it b}\/) we see that $f \alpha_1$ encircles the topmost hole of $P$, $f \alpha_2$ encircles no hole of $P$, and $f \alpha_3$ encircles the hole at the bottom of $P$. Thus if we use the symbol $\sim$ to denote that two curves are homologous in $P$, we have \[f \alpha_1 \sim \alpha_3, \ \ f \alpha_2 \sim 0, \ \ f \alpha_3 \sim \alpha_2.\] This is not very illuminating yet, but let us compute what happens with the iterates $f^2$ and $f^3$ of these loops. Applying $f$ to the above relations we get \[f^2 \alpha_1 \sim f \alpha_3 \sim \alpha_2, \ \ f^2 \alpha_2 \sim 0, \ \ f^2 \alpha_3 \sim f \alpha_2 \sim 0\] and similarly \[f^3 \alpha_1 \sim f \alpha_2 \sim 0, \ \ f^3 \alpha_2 \sim 0, \ \ f^3 \alpha_3 \sim 0.\] Now, in the same way that the $\alpha_i$ generate the homology of $P$, its images $f^3 \alpha_i$ generate the homology of $f^3 P$ since $f^3 |_P : P \longrightarrow f^3 P$ is a homeomorphism onto. The above relations tell us that each of the generators $f^3 \alpha_i$ is killed when included in $P$, or more formally that the inclusion $f^3P \subseteq P$ induces the zero homomorphism $H_1(f^3P) \stackrel{0}{\longrightarrow} H_1(P)$. For notational ease set $g := f^3$. Evidently $K = \bigcap_{n \geq 0} g^nP$ and so $\check{H}_1(K)$ is the inverse limit of the sequence \[H_1(P) \longleftarrow H_1(gP) \longleftarrow H_1(g^2P) \longleftarrow \ldots\] where each arrow is induced by the inclusion and is therefore $0$ as we have just seen. Thus the inverse limit of the sequence is zero, so $\check{H}_1(K) = 0$. Going backwards, the basin of attraction is the union of the increasing sequence $P \subseteq g^{-1}P \subseteq g^{-2}P \subseteq \ldots$ and so its homology is the direct limit of the sequence \[H_1(P) \longrightarrow H_1(g^{-1}P) \longrightarrow H_1(g^{-2}P) \longrightarrow \ldots\] where again each arrow is induced by the inclusion and is therefore the zero homomorphism. Hence $H_1(\mathcal{A}(K)) = 0$ and so the inclusion $K \subseteq \mathcal{A}(K)$ certainly induces an isomorphism in homology, even though neither the inclusion $K \subseteq P$ nor $P \subseteq \mathcal{A}(K)$ do.

\subsection{Our strategy of proof} The morale of the above examples is that in the discrete case ---in sharp contrast with the continuous case--- there may be no relation whatsoever between an attractor $K$ and its positively invariant neighbourhoods $P \subseteq \mathcal{A}(K)$. One way around this difficulty consists in \emph{assuming} that $K$ has a compact neighbourhood $P \subseteq \mathcal{A}(K)$ such that the inclusion $P \subseteq \mathcal{A}(K)$ is a homotopy equivalence. This assumption is essentially the one adopted in the papers by Gobbino and Mor\'on and Ruiz del Portal \cite{gobbino1,moronpaco1} and it implies that property (*) holds, so the same technique as in the continuous case can be applied.

Our approach to the problem is entirely different, roughly along the lines of the discussion of example (B) above. Instead of homology we will need to work with fundamental groups (and the higher homotopy groups) so we can use the Whitehead theorem to avoid having to construct explicitly a shape inverse for $i$. Much as we did earlier, we shall consider biinfinite sequences such as \[\ldots \longleftarrow \pi_1(f^{-1}P) \longleftarrow \pi_1(P) \longleftarrow \pi_1(P) \longleftarrow \ldots\] whose inverse limit is $\check{\pi}_1(K)$ and whose direct limit is $\pi_1(\mathcal{A}(K))$, and try to show that these two are isomorphic via $i_* : \check{\pi}_1(K) \longrightarrow \pi_1(\mathcal{A}(K))$. In Theorem \ref{teo:main} we shall prove that:
\begin{itemize}
	\item $i_*$ is injective as a \emph{geometric} consequence of the hypothesis that $K$ has polyhedral shape,
	\item $i_*$ is surjective as an \emph{algebraic} consequence of the fact that certain decrea{\-}sing sequence of subgroups \[\pi_1(P) \geq F_1 \geq F_2 \geq \ldots\] eventually stabilizes, which is closely related to the Mittag--Leffler property of ${\rm pro-}\pi_1(K)$.
\end{itemize}

After proving this in Section \ref{sec:main}, in Section \ref{sec:algebra} we point out a useful fact about the proof of Theorem \ref{teo:main}. Essentially, we observe that whenever $\pi_1(P)$ has certain algebraic properties, if $i_*$ is surjective then it is automatically injective. It so happens that, for dynamical systems on manifolds of low dimension, $P$ can always chosen to have those properties and it is exploting them that we arrive at Theorems \ref{teo:2mfds} and \ref{teo:3var}.

\section{Proof of Theorem \ref{teo:main}} \label{sec:main}

{\bf Notation.} Throughout this section $M$ will always be a locally compact absolute neighbourhood retract for metrizable spaces (ANR for short). If the reader is not familiar with ANRs, she may take an ANR to mean a locally contractible space, since both concepts are equivalent for finite dimensional spaces. References for this and some other results about ANRs that we shall need can be found in Appendix \ref{ap:shape}.
\medskip

We begin with two easy auxiliary propositions: Proposition \ref{prop:potencia} allows us to perform the useful operation of replacing $f$ by one of its powers; Proposition \ref{prop:basica} shows that we may assume that $K$ and $\mathcal{A}(K)$ are connected, which is convenient to apply Whitehead's theorem.

\begin{proposition} \label{prop:potencia} Let $K$ be an attractor for $f$ with basin of attraction $\mathcal{A}(K)$. Then $K$ is also an attractor for $f^N$ (for any $N \geq 1$) with the same basin of attraction.
\end{proposition}
\begin{proof} The diverse dynamical concepts are attached a subscript $f$ or $f^N$ in order to distinguish which homeomorphism they refer to.

Let $C$ be a compact subset of $\mathcal{A}_f(K)$ and $V$ a neighbourhood of $K$. There exists $n_0$ such that $f^n(C) \subseteq V$ for every $n \geq n_0$, so in particular $(f^N)^n(C) \subseteq V$ for every $n \geq n_0$ too. Thus $K$ is an attractor for $f^N$ whose basin of attraction $\mathcal{A}_{f^N}(K)$ contains $\mathcal{A}_f(K)$. Now, since $\mathcal{A}_f(K)$ is a neighbourhood of $K$, for every point $p \in \mathcal{A}_{f^N(K)}$ there exists $n_0$ such that $(f^N)^n(p) \in \mathcal{A}_f(K)$ for $n \geq n_0$. Using the fact that $\mathcal{A}_f(K)$ is invariant under $f$ we conclude that $p \in \mathcal{A}_f(K)$. Thus $\mathcal{A}_f(K) = \mathcal{A}_{f^N}(K)$.
\end{proof}

\begin{proposition} \label{prop:basica} Let $K$ be an attractor for $f$ with basin of attraction $\mathcal{A}(K)$. Then $K$ and $\mathcal{A}(K)$ have finitely many connected components, $K_1, \ldots, K_r$ and $A_1, \ldots, A_r$ respectively, with $K_i \subseteq A_i$ after relabeling. Moreover, there exists a suitable power $f^N$ of $f$ such that each $K_i$ is an attractor with basin of attraction $A_i$.
\end{proposition}
\begin{proof}[Proof of Proposition \ref{prop:basica}] Since an ANR is locally connected \cite[Theorem 7, p. 40]{mardesic1} each $p \in K$ has an open connected neighbourhood $U_p$ such that $\overline{U}_p$ is compact and contained in $\mathcal{A}(K)$. The compactness of $K$ implies that there exists a finite set $F \subseteq K$ such that $K \subseteq U := \bigcup_{p \in F} U_p$, so $U$ is a neighbourhood of $K$ whose closure $P := \overline{U} = \bigcup_{p \in F} \overline{U}_p$ is a compact subset of $\mathcal{A}(K)$. Each $\overline{U}_p$ is connected, being the closure of a connected set, so $P$ has at most $|F|$ connected components; let these be $P_1, \ldots, P_r$. Clearly every component of $P$ has nonempty intersection with $K$.

$P$ is a compact neighbourhood of the attractor $K$, so for some $n \geq 1$ we have $f^nP \subseteq P$, and consequently $f^nP_i \subseteq P_{j(i)}$ for suitable indices $j(i)$. Observe that, since $f^nP$ is a neighbourhood of $K$ as well and every component of $P$ meets $K$, every component of $P$ must meet $f^nP$; this implies that the $j(i)$ exhaust all of $\{1, \ldots, r\}$. Hence $i \mapsto j(i)$ is a permutation of $r$ elements and therefore ${(f^n)}^{r!} = f^{n \cdot r!}$ must send each $P_i$ into $P_i$. Let $N := n \cdot r!$ and replace $f$ by $f^N$ to keep notation simple, which entitles us to assume that $fP_i \subseteq P_i$ for all $1 \leq i \leq r$. Proposition \ref{prop:potencia} guarantees that $\mathcal{A}(K)$ does not change under this modification.

For every $1 \leq i \leq r$ let $K_i := \bigcap_{n \geq 1} f^n P_i$. The condition that $f P_i \subseteq P_i$ implies that the sequence $f^n P_i$ is decreasing, and since each of its members is compact and connected, we conclude that $K_i$ is compact and connected as well. Observe also that the $K_i$ are disjoint (because the $P_i$ are) and their union is precisely $\bigcap_{n \geq 1} f^n P = K$, so the $K_i$ are precisely the connected components of $K$.

Given $p \in P_i$, we have $f^n(p) \longrightarrow K$ because $P_i \subseteq \mathcal{A}(K)$. However, since $P_i$ is positively invariant we have $(f^n(p)) \subseteq P_i$, and since $K \bigcap P_i = K_i$ we conclude that $f^n(p) \longrightarrow K_i$. Thus $K_i$ is an attractor whose basin of attraction includes $P_i$, so in particular $\mathcal{A}(K_i) = \bigcup_{n \geq 0} f^{-n}P_i$.

The sets $\mathcal{A}(K_i)$ are connected, being the union of an increasing sequence of connected sets. They are disjoint, because $f^{-m}P_i \bigcap f^{-n}P_j = \emptyset$ if $i \neq j$ since (for example when $m \geq n$) we have \[f^m \left( f^{-m}P_i \bigcap f^{-n}P_j \right) = P_i \bigcap f^{m-n}P_j \subseteq P_i \bigcap P_j = \emptyset\] due to the positive invariance of $P_j$. Finally, $\mathcal{A}(K) = \bigcup_{n \geq 0} f^{-n} P$ is clearly the union of the $\mathcal{A}(K_i)$, and so these are precisely the components of $\mathcal{A}(K)$.
\end{proof}

The basin of attraction of an attractor is always an invariant open subset of the phase space $M$. Therefore it is locally compact and metrizable, because $M$ has these properties. Similarly, since an open subset of an ANR is again an ANR, it follows that $\mathcal{A}(K)$ is an ANR. Thus in all our results we may replace $M$ with $\mathcal{A}(K)$ and $f$ with its restriction $f|_{\mathcal{A}(K)}$; that is, we may simply assume that $M = \mathcal{A}(K)$. Similarly, Proposition \ref{prop:basica} entitles us to assume that $K$ and therefore also $\mathcal{A}(K)$ are connected. Summing up:

\begin{remark} \label{rem:simple} In this section we shall always assume, without loss of generality, that $K$ is a global attractor (that is, its basin of attraction $\mathcal{A}(K)$ is all of $M$) and both $K$ and $M$ are connected. Similar assumptions will be made without further explanation, when convenient, throughout the paper.
\end{remark}

Now we can start with the proof of Theorem \ref{teo:main} proper.

\begin{proof}[Proof of Theorem \ref{teo:main}] ($\Rightarrow$) Assume first that $i : K \subseteq \mathcal{A}(K) = M$ is a shape equivalence. We want to prove that $K$ has the shape of a polyhedron. This is very easy. Since $M$ is an ANR, it has the homotopy type of a polyhedron $Q$ (this is a general property of ANRs) and in particular it has the shape of $Q$ because every homotopy equivalence is a shape equivalence. Since by assumption $i$ is a shape equivalence, it follows that $K$ also has the shape of the polyhedron $Q$.

($\Leftarrow$) Let us assume now that $K$ has the shape of a polyhedron $Q$. We want to prove that the inclusion $i$ is a shape equivalence. We remind the reader that, according to Remark \ref{rem:simple} we assume that $K$ and $M$ are connected and $K$ is a global attractor in $M$.

Fix any basepoint $* \in K$ and denote $i : (K,*) \subseteq (M,*)$ the inclusion. For each dimension $d \geq 1$ there is an induced homomorphism $i_* : \check{\pi}_d(K,*) \longrightarrow \check{\pi}_d(M,*)$. Since $M$ is an ANR we may (and will) identify $\check{\pi}_d(M,*)$ with the ordinary homotopy group $\pi_d(M,*)$.
\medskip

{\it Claim 1.} It is enough to prove that $i_* : \check{\pi}_d(K,*) \longrightarrow \pi_d(M,*)$ is an isomorphism for each $d \geq 1$.

{\it Proof of claim.} Since $K$ has the shape of a polyhedron, $(K,*)$ has the pointed shape of a pointed polyhedron \cite[Theorem 1, p. 219]{mardesic1}. This means that there exist a pointed polyhedron $(Q,*)$ and a shape equivalence $q : (Q,*) \longrightarrow (K,*)$. Notice that $q_* : \pi_d(Q,*) \longrightarrow \check{\pi}_d(K,*)$ is an isomorphism for each $d \geq 1$, where once again we have identified $\check{\pi}_d(Q,*)$ with $\pi_d(Q,*)$.

Now $iq : (Q,*) \longrightarrow (M,*)$ induces isomorphisms $(iq)_* = i_* q_* : \pi_d(Q,*) \longrightarrow \pi_d(M,*)$. Thus by the classical Whitehead theorem \cite[Corollary 24, p. 405]{spanier1} we see that $iq$ is a homotopy equivalence, and consequently a shape equivalence. Since $q$ is a shape equivalence, it follows that $i$ is a shape equivalence too. $_{\blacksquare}$
\medskip

{\it Claim 2.} $K$ has arbitrarily small neighbourhoods $U$ with the following three properties: ({\it i}\/) $U$ is a connected ANR, ({\it ii}\/) $\overline{U}$ is compact and ({\it iii}\/) possibly after replacing $f$ by a suitable power of itself, $fU \subseteq U$.

{\it Proof of claim.} Indeed, let $V$ be a neighbourhood of $K$. Since $M$ is locally compact, we may reduce $V$ so that $\overline{V}$ is compact. Let $U \subseteq V$ be an open neighbourhood of $K$, and discard all of its components except for the one that contains the connected set $K$. Then $U$ is an ANR, because it is open in the ANR $M$, and it satisfies ({\it i}\/) and ({\it ii}\/). Now, since $\overline{U}$ is compact, there exists $N$ such that $f^N \overline{U} \subseteq U$. Replacing $f$ by $f^N$ we achieve all three conditions. $_{\blacksquare}$
\medskip

Let $U$ be a neighbourhood of $K$ as in Claim 2. Condition ({\it iii}\/) implies that $f^{n+1}U \subseteq f^nU$ for every $n \in \mathbb{Z}$. Denote by $j_{n+1} : \pi_d(f^{n+1}U,*) \longrightarrow \pi_d(f^nU,*)$ the inclusion induced homomorphisms and consider the biinfinite sequence $\mathcal{U}$ \[  \xymatrix{\ldots & \pi_d(f^{-1}U,*) \ar[l]_-{j_{-1}} & \pi_d(U,*) \ar[l]_-{j_0} & \pi_d(fU,*) \ar[l]_-{j_1} & \pi_d(f^2U,*) \ar[l]_-{j_2} & \ldots \ar[l]_-{j_3}}\]

Condition ({\it ii}\/) on $U$ implies that $\bigcap_{n \geq 1} f^n U = K$. Since $U$, and consequently each $f^nU$, is an ANR, it follows that the inverse limit of the above sequence is precisely $\check{\pi}_d(K,*)$. On the other hand, the interiors of the $f^nU$ form an increasing (for decreasing $n$) sequence whose union is all of $M$; therefore $\pi_d(M,*)$ is the direct limit of the above sequence.
\medskip

{\it Claim 3.} Let $W$ be a connected open neighbourhood of $K$ in $M$. Assume that for some basepoint $\star \in K$ the inclusion $(K,\star) \subseteq (W,\star)$ induces an injective homomorphism $\check{\pi}_d(K,\star) \longrightarrow \pi_d(W,\star)$. Then the same holds true when the basepoint $\star$ is replaced by $*$.

{\it Proof of claim.} Let $W_n$ be a decreasing basis of connected open neighbourhoods of $K$ in $M$, all contained in $W$. There exists a sequence of paths $\gamma_n$ such that ({\it i}\/) $\gamma_n \subseteq W_n$, ({\it ii}\/) $\gamma_n$ joins $\star$ and $*$, ({\it iii}\/) $\gamma_{n+1} \simeq \gamma_n$ in $W_n$ relative to their endpoints. This is true because $K$ is \emph{joinable} \cite[p. 144]{krasinkiewiczminc1}, which in turn is a consequence of $K$ having the pointed shape of a pointed polyhedron \cite[Proposition 1.8, p. 145]{krasinkiewiczminc1}. Denote $h_{\gamma_n} : \pi_d(W_n,*) \longrightarrow \pi_d(W_n,\star)$ the standard basepoint change isomorphism \cite[Theorem 8, p. 384]{spanier1}.

Let $\alpha \in \check{\pi}_d(K,*)$, and assume that $\alpha = 1$ in $\pi_d(W,*)$. Since $\check{\pi}_d(K,*)$ is the inverse limit of \[\pi_d(W_1,*) \longleftarrow \pi_d(W_2,*) \longleftarrow \pi_d(W_3,*) \longleftarrow \ldots\] where each arrow is induced by inclusion, $\alpha$ is represented by a sequence $(\alpha_n)$ where each $\alpha_n \in \pi_d(W_n,*)$ and $\alpha_{n+1} \simeq \alpha_n$ (rel. $*$) in $W_n$. The condition that $\alpha= 1$ in $\pi_d(W,*)$ implies that $\alpha_n \simeq {\rm ct}_*$ (rel. $*$) in $W$ for every $n \in \mathbb{N}$. Here ${\rm ct}_*$ means the map constantly $*$. We warn the reader that we are making no notational distinction between a map and the homotopy class it represents.

By ({\it iii}\/) we see that $h_{\gamma_{n+1}} \alpha_{n+1} \simeq h_{\gamma_n}\alpha_n$ (rel. $\star$) in $W_n$, and so $\alpha' := (h_{\gamma_n} \alpha_n)$ is a legitimate element of $\check{\pi}_d(K,\star)$. Since each $\alpha_n \simeq {\rm ct}_*$ (rel. $*$) in $W$, clearly $h_{\gamma_n} \alpha_n \simeq {\rm ct}_{\star}$ (rel. $\star$) in $W$ for each $n$, and so $\alpha' = 1$ in $\pi_d(W,\star)$. The assumption that the inclusion $(K,\star) \subseteq (W,\star)$ induces an injective homomorphism now implies that $\alpha' = 1$ in $\check{\pi}_d(K,\star)$, so $h_{\gamma_n} \alpha_n \simeq {\rm ct}_{\star}$ (rel. $\star$) in $W_n$ for each $n$. Therefore $\alpha_n \simeq {\rm ct}_*$ (rel. $*$) in $W_n$ for each $n$ and it follows that $\alpha = 1$ in $\check{\pi}_d(K,*)$. $_{\blacksquare}$
\medskip

Let us prove now that $i_*$ is injective. Since $(K,*)$ has the (pointed) shape of a polyhedron, it is a pointed fundamental absolute neighbourhood retract \cite[Theorem 15, p. 234]{mardesic1}. This means that $K$ has a neighbourhood $V$ in $M$ such that the inclusion $(K,*) \subseteq (V,*)$ has a left inverse $r$ in the shape category. Let $U \subseteq V$ be as in Claim 2 and denote $i_0 : (K,*) \subseteq (U,*)$ the inclusion. More generally, for each $n \in \mathbb{Z}$ let $i_n : (K,*) \subseteq (f^n U,*)$ be the inclusion.

According to sequence $\mathcal{U}$, to prove that $i_* : \check{\pi}_d(K,*) \longrightarrow \pi_d(M,*)$ is injective it is enough to show that each $i_n$ induces an injective homomorphism $i_{n,*} : \check{\pi}_d(K,*) \longrightarrow \pi_d(f^nU,*)$. The case $n = 0$ is immediate: from the fact that the composition \[(K,*) \stackrel{i_0}{\subseteq} (U,*) \subseteq (V,*) \stackrel{r}{\longrightarrow} (K,*)\] equals the identity in the shape category we see that $i_{0,*} : \check{\pi}_d(K,*) \longrightarrow \pi_d(U,*)$ is injective. To settle the case $n \neq 0$ we need Claim 3. Let $n \in \mathbb{Z}$. The composition \[(K,f^n*) \stackrel{f^{-n}}{\longrightarrow} (K,*) \stackrel{i_0}{\subseteq} (U,*) \stackrel{f^n}{\longrightarrow} (f^nU,f^n*)\] is precisely the inclusion $(K,f^n*) \subseteq (f^nU,f^n*)$. Therefore it induces an injective homomorphism $\check{\pi}_d(K,f^n*) \longrightarrow \pi_d(f^nU,f^n*)$ because it equals the composition $f^{-n}_* i_{0,*} f^n_*$ and we saw above that $i_{0,*}$ is injective. By Claim 3 with $W = f^nU$ and $\star = f^n*$ we conclude that $i_{n,*}$ is injective.

By Claim 1 it only remains to show that $i_*$ is surjective. Let $U$ be a neighbourhood of $K$ as in Claim 2 and consider once again the sequence $\mathcal{U}$. Denote $G_n := {\rm im}\ j_{n+1}$. It is easy to see that the direct and inverse limits of $\mathcal{U}$ do not change if we replace each of its terms by the image of the preceding one; that is, the sequence \[\xymatrix@=15mm{ \ldots & G_{-1} \ar[l]_-{j_{-1}|_{G_{-1}}} & G_0 \ar[l]_-{j_0|_{G_0}} & G_1 \ar[l]_-{j_1|_{G_1}} & G_2 \ar[l]_-{j_2|_{G_2}} & \ldots \ar[l]_-{j_3|_{G_3}}}\] has $\pi_d(M,*)$ and $\check{\pi}_d(K,*)$ as its direct and inverse limits, respectively. To prove that $i_*$ is surjective it will be enough to show that each bonding morphism in this sequence is surjective. Explicitly, we need to show that \[(S_n) : j_{n+1}(G_{n+1}) = G_n\] is true for every $n \in \mathbb{Z}$.

We start with the case $n = 0$. Since $(K,*)$ has the pointed shape of a pointed polyhedron, ${\rm pro-}\pi_d(K,*)$ has the Mittag--Leffler property. Thus there exists some $N \geq 1$ such that ${\rm im}\ (j_1 \circ \ldots \circ j_n) = {\rm im}\ (j_1 \circ \ldots \circ j_N)$ for every $n \geq N$, and replacing $f$ by $f^N$ we are entitled to assume that $N = 1$. Observe that then \[G_0 = {\rm im}\ j_1 = {\rm im}\ (j_1 \circ j_2) = j_1({\rm im}\ j_2) = j_1(G_1),\] which proves ($S_n$) for $n = 0$.

Now let $n \in \mathbb{Z}$ and consider the commutative diagram \[\xymatrix{(f^n U,*) \ar[d]_f & (f^{n+1} U,*) \ar[d]_f \ar[l] & (f^{n+2} U,*) \ar[d]_f \ar[l] \\ (f^{n+1} U,f*) & (f^{n+2} U,f*) \ar[l] & (f^{n+3} U,f*) \ar[l]}\] where the unlabeled arrows are inclusions. Since $U$ is a connected ANR, it is path connected and so is $f^{n+3} U$, being homeomorphic to $U$. Thus there is a path $\gamma_n$ in $f^{n+3} U$ that joins $*$ and $f*$, because both points belong to $K \subseteq f^{n+3} U$. Denote $h_{{\gamma_n}}$ the standard basepoint change isomorphism from $f*$ to $*$, as in Claim 3. Then there is a commutative diagram \[\xymatrix@C=20mm{\pi_d(f^n U,*) \ar[d]_{h_{{\gamma_n}} f_*}^{\cong} & \pi_d(f^{n+1} U,*) \ar[l]_{j_{n+1}} \ar[d]_{h_{{\gamma_n}} f_*}^{\cong} & \pi_d(f^{n+2} U,*) \ar[l]_{j_{n+2}} \ar[d]_{h_{{\gamma_n}} f_*}^{\cong} \\ \pi_d(f^{n+1} U,*) & \pi_d(f^{n+2} U,*) \ar[l]_{j_{n+2}} & \pi_d(f^{n+3} U,*) \ar[l]_{j_{n+3}}}\]

Since each $h_{{\gamma_n}} f_*$ is an isomorphism because $f$ is a homeomorphism onto its image, it follows from the commutativity of the left square that $h_{{\gamma_n}} f_*$ takes $G_n$ onto $G_{n+1}$, and similarly (from the right square) that it also takes $G_{n+1}$ onto $G_{n+2}$. Concentrating on the left square of the above diagram alone and replacing each group by the image of the arrow that enters it, we obtain \[\xymatrix@C=20mm{G_n \ar[d]_{h_{{\gamma_n}} f_*|_{G_n}}^{\cong} & G_{n+1} \ar[l]_{j_{n+1}|_{G_{n+1}}} \ar[d]^{h_{{\gamma_n}} f_*|_{G_{n+1}}}_{\cong} \\ G_{n+1} & G_{n+2} \ar[l]_{j_{n+2}|_{G_{n+2}}}}\] and it is then obvious that $j_{n+1}|_{G_{n+1}}$ is surjective $\Leftrightarrow$ $j_{n+2}|_{G_{n+2}}$ is surjective. That is, $j_{n+1}(G_{n+1}) = G_n \Leftrightarrow j_{n+2}(G_{n+2}) = G_{n+1}$. Since we showed earlier that $j_1(G_1) = G_0$, induction proves ($S_n$) for every $n \in \mathbb{Z}$ (use $\Rightarrow$ for $n \geq 1$ and $\Leftarrow$ for $n \leq -1$). This concludes the proof of Theorem \ref{teo:main}.
\end{proof}

To finish this section we prove Corollary \ref{cor:fixed}. First we need a purely shape theoretical remark:

\begin{remark} \label{rem:dominated} Suppose a compact set $K$ is shape dominated by a polyhedron $P$. This means that there exist two shape morphisms $j : K \longrightarrow P$ and $k : P \longrightarrow K$ such that $kj = {\rm id}_K$ in the shape category. Then $K$ actually has the shape of a polyhedron (different from $P$, in general). As a consequence of this and Theorem \ref{teo:main}, if an attractor $K$ is shape dominated by a polyhedron then the inclusion $i : K \subseteq \mathcal{A}(K)$ is a shape equivalence.
\end{remark}
\begin{proof} Since $K$ is shape dominated by a polyhedron, it is an absolute shape neighbourhood retract \cite[Theorem 12, p. 233]{mardesic1}. Since $K$ is compact, this implies that it is a fundamental absolute neighbourhood retract \cite[Corollary 2, p. 234]{mardesic1} and consequently a pointed fundamental absolute neighbourhood retract \cite[Theorem 19, p. 236]{mardesic1}. It follows that $K$ has the shape of a polyhedron \cite[Theorem 15, p. 234]{mardesic1}.
\end{proof}

\begin{proof}[Proof of Corollary \ref{cor:fixed}] The assumption is that $f|_K$ is homotopic to the identity ${\rm id}_K$. It suffices to prove that there exist an open neighbourhood $V$ of $K$ and a shape morphism $r : V \longrightarrow K$ such that $rj = {\rm id}_K$, where $j : K \subseteq V$ is the inclusion. Indeed, admit for a second that we have already proved this. Since $V$ is open in the ANR $M$, it is an ANR itself and so it has the homotopy type of a polyhedron $P$. Let $u : V \longrightarrow P$ be a homotopy equivalence and $v : P \longrightarrow V$ its (homotopy) inverse, so that $vu = {\rm id}_V$. Inserting this in $rj = {\rm id}_K$ we get $(rv)(uj) = {\rm id}_K$, so we see that the shape morphism $uj : K \longrightarrow P$ has a left inverse. Thus $K$ is shape dominated by the polyhedron $P$ and due to Remark \ref{rem:dominated} we may apply Theorem \ref{teo:main} to conclude that the inclusion $i : K \subseteq \mathcal{A}(K)$ is a shape equivalence.

Let us construct, then, $V$ and $r$. Let $H : K \times [0,1] \longrightarrow K$ be a homotopy such that $H_0 = {\rm id}_K$ and $H_1 = f|_K$, and let $U$ be an open neighbourhood of $K$ as in Claim 2 of the proof of Theorem \ref{teo:main}. It is easy to see that there exist an open neighbourhood $V$ of $K$ and an extension $\tilde{H} : V \times [0,1] \longrightarrow U$ of $H$ such that $\tilde{H}_0 = {\rm id}|_V$ and $\tilde{H}_1 = f|_V$ \cite[Theorem 8, p. 40]{mardesic1}. Notice that the first condition implies that $V \subseteq U$. Consider the ANR expansion of $K$ given by \[\xymatrix{U & fU \ar[l] & f^2U \ar[l] & \ldots \ar[l]}\] and the sequence of maps $r_k := f^k|_V : V \longrightarrow f^k U$. For each $k = 0,1,2,\ldots$ the map $f^k \tilde{H} : V \times [0,1] \longrightarrow f^kU$ provides a homotopy between $f^k|_V$ and $f^{k+1}|_V$, so the diagram \[\xymatrix@=20mm{U & fU \ar[l] & f^2U \ar[l] & \ldots \ar[l] \\ V \ar[u]^{{\rm id}|_V} \ar[ur]^{f|_V} \ar[urr]^{f^2|_V} & & } \] is commutative up to homotopy and consequently $(r_k)$ defines a shape morphism $r = (r_k) : V \longrightarrow K$. Since $r_k j = f^k|_K \simeq {\rm id}_K$, it follows that $r j = {\rm id}_K$ in the shape category. This concludes the proof.
\end{proof}

\section{Abstracting the algebra behind the proof of Theorem \ref{teo:main}} \label{sec:algebra}

Looking back at the proof of Theorem \ref{teo:main} it is apparent that the surjectivity of $i_*$ was established using a mostly algebraic argument involving the Mittag--Leffler property and the very special structure of sequence $\mathcal{U}$. It turns out that, under suitable conditions, this argument can be refined to yield also the injectivity of $i_*$. Thus it is worthwhile to abstract the algebra behind the argument, and that is what this section is devoted to. The results obtained here will be used in the following two sections to establish Theorems \ref{teo:2mfds} and \ref{teo:3var}.
\medskip

Consider the sequence \[ \mathcal{G} : \xymatrix{\ldots & G_{-1} \ar[l]_-{\varphi_{-1}} & G_0 \ar[l]_-{\varphi_0} & G_1 \ar[l]_-{\varphi_1} & \ldots \ar[l]_-{\varphi_2}} \] where each $G_n$ is a (possibly non abelian) group and each $\varphi_n : G_n \longrightarrow G_{n-1}$ is a group homomorphism. $\mathcal{G}$ has an inverse limit $\underline{G}$ endowed with natural maps from $\underline{G}$ to each $G_n$; we single out one, namely $\underline{\varphi} : \underline{G} \longrightarrow G_0$. $\mathcal{G}$ also has a direct limit $\overline{G}$ endowed with natural maps from each $G_n$ to $\overline{G}$; we single out one, namely $\overline{\varphi} : G_0 \longrightarrow \overline{G}$. Consider the decreasing sequence of subgroups of $G_0$ defined by $F_0 := G_0$ and $F_n := {\rm im}\ (\varphi_1 \ldots \varphi_n)$ for $n \geq 1$. Clearly \[(S) : G_0 = F_0 \geq F_1 \geq F_2 \geq \ldots\] and we shall say that $(S)$ \emph{stabilizes} if there exists $n \geq 0$ such that $F_n = F_{n+1} = F_{n+2} = \ldots$ On occasion we will need to be more precise and say that $(S)$ stabilizes \emph{at} $F_n$.

In our applications the groups $G_n$ will be homology or homotopy groups of iterates $f^nP$. The following definition is an algebraic abstraction of the idea that each $f^{n+1}P$ lies in $f^nP$ in the same way that $fP$ lies in $P$.

\begin{definition} Let us say that $\mathcal{G}$ is \emph{rigid} if the following condition holds: for every $m,n \in \mathbb{Z}$, $m \geq 1$, there is a commutative diagram \[\xymatrix@=15mm{G_0 \ar[d]_{\cong} & G_1 \ar[l]_{\varphi_1} \ar[d]_{\cong} & \ldots \ar[l]_{\varphi_2} \ar@{}[d]|{\vdots} & G_m \ar[l]_{\varphi_m} \ar[d]_{\cong} \\ G_n & G_{n+1} \ar[l]^{\varphi_{n+1}} & \ldots \ar[l]^{\varphi_{n+2}} & G_{n+m} \ar[l]^{\varphi_{n+m}}}\] where the vertical arrows are isomorphisms.
\end{definition}

\begin{proposition} \label{prop:one} Let $\mathcal{G}$ be rigid. The following properties hold.
\begin{enumerate}
	\item All the $G_n$ are isomorphic to each other.
	\item If $F_n = F_{n+1}$ for some $n$, then $F_n = F_{n+1} = F_{n+2} = \ldots$ so that $(S)$ stabilizes.
	\item $(S)$ stabilizes if, and only if, \[ \xymatrix{ G_0 & G_1 \ar[l]_-{\varphi_1} & \ldots \ar[l]_-{\varphi_2}}\] is Mittag--Leffler.
\end{enumerate}
\end{proposition}
\begin{proof} (1) This is trivial.

(2) Since $\mathcal{G}$ is rigid, there exists a commutative diagram \[\xymatrix@=15mm{G_0 \ar[d]_{\cong}^{\psi_0} & \ldots \ar[l]_{\varphi_1} & G_n \ar[l]_{\varphi_n} \ar[d]_{\cong}^{\psi_n} & G_{n+1} \ar[l]_{\varphi_{n+1}} \ar[d]_{\cong}^{\psi_{n+1}} \\ G_1 & \ldots \ar[l]_{\varphi_2} & G_{n+1} \ar[l]_{\varphi_{n+1}} & G_{n+2} \ar[l]_{\varphi_{n+2}}}\]

Since $\psi_0 \varphi_1 \ldots \varphi_n = \varphi_2  \ldots  \varphi_{n+1}  \psi_n$, evaluating both sides of the equality on $G_n$ and taking into account that $\psi_n(G_n) = G_{n+1}$ because $\psi_n$ is onto it follows that $\psi_0(F_n) = (\varphi_2  \ldots  \varphi_{n+1})(G_{n+1})$. Similarly $\psi_0(F_{n+1}) = (\varphi_2  \ldots  \varphi_{n+2})(G_{n+2})$. Since $F_n = F_{n+1}$, we see that $(\varphi_2  \ldots  \varphi_{n+1})(G_{n+1}) = (\varphi_2  \ldots  \varphi_{n+2})(G_{n+2})$ and applying $\varphi_1$ to both sides of the equality yields $F_{n+1} = F_{n+2}$. Similar arguments show that $F_{n+1} = F_{n+2} = \ldots$

(3) The proof goes along the lines of part (2) and is left to the reader.
\end{proof}

Proposition \ref{prop:one}.2 implies that $(S)$ has either the structure \[G_0 = F_0 \gneqq F_1 \gneqq \ldots \gneqq F_n = F_{n+1} = \ldots \ \ \ \ \text{(if it stabilizes)}\] or \[G_0 = F_0 \gneqq F_1 \gneqq F_2 \gneqq \ldots\ \ \ \ \text{(if it does not stabilize).}\]

Now we recall two algebraic notions that will be key for our argument. A group $G$ is \emph{residually finite} if for every $g \in G \backslash \{1\}$ there exists a homomorphism $\alpha_g$ of $G$ to a finite group $F_g$ such that $\alpha_g(g) \neq 1$. It follows immediately from the definition that subgroups and direct sums of residually finite groups are again residually finite.

\begin{example} Finite groups are trivially residually finite. A finitely generated abelian group $G$ decomposes as a direct sum of $\mathbb{Z}$ and $\mathbb{Z}_n$ summands, which are easily seen to be residually finite; hence $G$ is residually finite as well.
\end{example}

A group $G$ is \emph{Hopfian} if every epimorphism $\alpha : G \longrightarrow G$ is an isomorphism. The following remark is trivial, but worth noting:

\begin{remark} If $G$ and $H$ are isomorphic groups, $G$ is Hopfian and $\alpha : G \longrightarrow H$ is an epimorphism, then $\alpha$ is an isomorphism.
\end{remark}

We will also need a deep result of Malcev: a finitely generated and residually finite group is Hopfian \cite[Theorem 4.10, p. 197]{lyndonschupp1}. With this we can prove the main result of this section, which is the following theorem.

\begin{theorem} \label{teo:algebra} Let $\mathcal{G}$ be as above. Assume that
\begin{enumerate}
	\item[({\it a}\/)] $\mathcal{G}$ is rigid.
	\item[({\it b}\/)] $(S)$ stabilizes.
	\item[({\it c}\/)] $G_0$ is residually finite and finitely generated.
\end{enumerate}
Then the map $\overline{\varphi}  \underline{\varphi} : \underline{G} \longrightarrow \overline{G}$ is an isomorphism. More precisely, if $(S)$ stabilizes at $F_n$ then the two arrows in the following diagram are isomorphisms: \[\xymatrix{\underline{G} \ar[r]^{\underline{\varphi}} & F_n \ar[r]^{\overline{\varphi}|_{F_n}} & \overline{G}}\]
\end{theorem}
\begin{proof} Since $(S)$ stabilizes, $F_n = F_{n+1}$ for some $n$. We leave it to the reader to show that there is no loss in generality in assuming, as we will, that $F_1 = F_2$.

Denote $G'_n := {\rm im}\ \varphi_{n+1} = \varphi_{n+1}(G_{n+1})$. If we replace each term of $\mathcal{G}$ by the image of the bonding morphism entering it, and the bonding morphisms by the obvious restrictions, we get the sequence $\mathcal{G'}$ \[\xymatrix@=15mm{ \ldots & G'_{-1} \ar[l]_{\varphi_{-1}|_{G'_{-1}}} & G'_0 \ar[l]_{\varphi_0|_{G'_0}} & G'_1 \ar[l]_{\varphi_1|_{G'_1}} & G'_2 \ar[l]_{\varphi_2|_{G'_2}} & \ldots \ar[l]_{\varphi_3|_{G'_3}}} \] and it is not difficult to see that its inverse and direct limits $\underline{G'}$ and $\overline{G'}$ can be identified with those of $\mathcal{G}$, namely $\underline{G}$ and $\overline{G}$. Under these identifications the natural map from $\underline{G'}$ to $G'_0$ is precisely $\underline{\varphi} : \underline{G} \longrightarrow G'_0$ and the natural map from $G'_0$ to $\overline{G'}$ is precisely $\overline{\varphi}|_{G'_0} : G'_0 \longrightarrow \overline{G}$. Thus to prove the theorem it will be enough to show that $\varphi_n|_{G'_n}$ is an isomorphism for each $n \in \mathbb{Z}$.

Fix some $n \in \mathbb{Z}$. The rigidity condition on $\mathcal{G}$ implies that there is a commutative diagram \[\xymatrix@=15mm{G_0 \ar[d]^{\psi_0} & G_1 \ar[l]_{\varphi_1} \ar[d]^{\psi_1} & G_2 \ar[l]_{\varphi_2} \ar[d]^{\psi_2} \\ G_n & G_{n+1} \ar[l]^{\varphi_{n+1}} & G_{n+2} \ar[l]^{\varphi_{n+2}}}\] where $\psi_0$, $\psi_1$ and $\psi_2$ are isomorphisms. Since $\psi_0  \varphi_1 = \varphi_{n+1}  \psi_1$, it follows that $\psi_0  \varphi_1(G_1) = \varphi_{n+1}  \psi_1(G_1)$ so $\psi_0(G'_0) = \varphi_{n+1}(G_{n+1}) = G'_n$, and similarly $\psi_1(G'_1) = G'_{n+1}$. Therefore we have a commutative diagram \[\xymatrix@=15mm{G'_0 \ar[d]^{{\psi_0}|_{G'_0}}_{\cong} & G'_1 \ar[l]_{{\varphi_1}|_{G'_1}} \ar[d]^{{\psi_1}|_{G'_1}}_{\cong} \\ G'_n & G'_{n+1} \ar[l]^{\varphi_{n+1}|_{G'_{n+1}}} }\] where the vertical arrows are still isomorphisms. Since the upper arrow is surjective because \[\varphi_1(G'_1) = \varphi_1(\varphi_2(G_2)) = (\varphi_1  \varphi_2)(G_2) = \varphi_1(G_1) = G'_0,\] the third equality being a consequence of assumption ({\it b}\/), it follows that the lower arrow is surjective too. Observe also that $G'_n$ is isomorphic to $G'_0$; since this is true for all $n$, all the $G'_n$ are isomorphic to each other.

$G'_0$ is residually finite because it is a subgroup of the residually finite group $G_0$. Also, since $G_1$ is isomorphic to $G_0$ by rigidity, $G_1$ is finitely generated and consequently $G'_0 = \varphi_1(G_1)$ is finitely generated too. Thus by Malcev's theorem $G'_0$ is Hopfian, and as a consequence all the $G'_n$ are Hopfian. Then each map $\varphi_n|_{G'_n}$ is an epimorphism between two isomorphic Hopfian groups, and therefore they are all isomorphisms.
\end{proof}

Notice that $F_n$, being defined as $(\varphi_1 \ldots \varphi_n)(G_n)$, is both a subgroup of $G_0$ and a quotient of $G_n$. If $\mathcal{G}$ is rigid then $G_n$ is isomorphic to $G_0$, so $F_n$ is also a quotient of $G_0$. This proves the following:

\begin{remark} \label{rem:fg} Under the hypotheses of Theorem \ref{teo:algebra}, $\underline{G}$ and $\overline{G}$ are isomorphic to both a subgroup and a quotient of $G_0$.
\end{remark}

In order to illustrate how the abstract algebraic setting just described fits our dynamical picture we shall prove Theorem \ref{teo:hom} below, which is a homological counterpart of Theorem \ref{teo:main}. Notice how there are no hypotheses on $K$ whatsoever.

\begin{theorem} \label{teo:hom} Let $K$ be an attractor in a locally compact, metrizable ANR $M$. Then the inclusion $i : K \subseteq \mathcal{A}(K)$ induces isomorphisms $i_* : \check{H}_d(K;\mathbb{Z}_2) \longrightarrow H_d(\mathcal{A}(K);\mathbb{Z}_2)$ for all $d \geq 0$. Moreover, the groups $\check{H}_d(K;\mathbb{Z}_2)$ are all finite.
\end{theorem}
\begin{proof} We prove the theorem under the additional assumption \[(A) : \text{\emph{$K$ has a neighbourhood $P \subseteq \mathcal{A}(K)$ that is a compact ANR}}\] which simplifies our argument but still covers most of the cases of interest. The general case can be reduced to this one using a detour into infinite dimensions as described in Appendix \ref{app:trick}.

After replacing $f$ by a suitable power of itself we may assume that $fP \subseteq P$, so $f^{n+1}P \subseteq f^nP$ for every $n \in \mathbb{Z}$. Now fix a dimension $d \geq 0$. Denote by $\varphi_{n+1} : H_d(f^{n+1}P;\mathbb{Z}_2) \longrightarrow H_d(f^nP;\mathbb{Z}_2)$ the inclusion induced homomorphisms and consider the biinfinite sequence \[\mathcal{H} : \xymatrix{\ldots & H_d(f^{-1}P;\mathbb{Z}_2) \ar[l]_-{\varphi_{-1}} & H_d(P;\mathbb{Z}_2) \ar[l]_-{\varphi_0} & H_d(fP;\mathbb{Z}_2) \ar[l]_-{\varphi_1} & \ldots \ar[l]_-{\varphi_2}}\]

The inverse limit of $\mathcal{H}$ is precisely $\check{H}_d(K;\mathbb{Z}_2)$, and its direct limit is $H_d(\mathcal{A}(K);\mathbb{Z}_2)$. It is not difficult to check that $i_* : \check{H}_d(K;\mathbb{Z}_2) \longrightarrow H_d(\mathcal{A}(K);\mathbb{Z}_2)$ is precisely $\overline{\varphi} \underline{\varphi}$, so to prove the theorem we only need to check that the hypotheses of Theorem \ref{teo:algebra} are met.
\medskip

({\it a}\/) To check that $\mathcal{H}$ is rigid, let $m,n \in \mathbb{Z}$, $m \geq 1$, and consider the following commutative diagram: \[\xymatrix@=15mm{H_d(P;\mathbb{Z}_2) \ar[d]_{(f^n|_{P})_*} & H_d(fP;\mathbb{Z}_2) \ar[l]_{\varphi_1} \ar[d]_{(f^n|_{fP})_*} & \ldots \ar[l]_{\varphi_2} \ar@{}[d]|{\vdots} & H_d(f^mP;\mathbb{Z}_2) \ar[l]_{\varphi_m} \ar[d]_{(f^n|_{f^mP})_*} \\ H_d(f^nP;\mathbb{Z}_2) & H_d(f^{n+1}P;\mathbb{Z}_2) \ar[l]^{\varphi_{n+1}} & \ldots \ar[l]^{\varphi_{n+2}} & H_d(f^{n+m}P;\mathbb{Z}_2) \ar[l]^{\varphi_{n+m}}}\] The vertical arrows are isomorphisms, because $f^n|_P, f^n|_{fP}, \ldots, f^n|_{f^mP}$ are homeomorphisms onto their images. Hence $\mathcal{H}$ is rigid.
\smallskip

({\it b}\/) Each inclusion $f^nP \subseteq P$ induces a homomorphism $H_d(f^nP;\mathbb{Z}_2) \longrightarrow H_d(P;\mathbb{Z}_2)$ whose image is precisely $F_n = {\rm im}\ (\varphi_1 \ldots \varphi_n)$, and we need to check that the sequence \[(S) : H_d(P;\mathbb{Z}_2) \supseteq F_1 \supseteq F_2 \supseteq \ldots\] stabilizes. Here we use assumption $(A)$: since $P$ is a compact ANR, $H_d(P;\mathbb{Z}_2)$ is finitely generated and therefore finite, so obviously $(S)$ stabilizes.
\smallskip

({\it c}\/) We have just seen that $H_d(P;\mathbb{Z}_2)$ is finite. Thus it is clearly finitely generated and residually finite.
\medskip
 
Theorem \ref{teo:algebra} immediately implies that $i_*$ is an isomorphism. Also, Remark \ref{rem:fg} shows that $\check{H}_d(K;\mathbb{Z}_2)$, being a subgroup of the finite group $H_d(P;\mathbb{Z}_2)$, is finite.
\end{proof}

Joint work with Professor F. R. Ruiz del Portal, motivated by a suggestion of his, has led to a neat generalization of Theorem \ref{teo:hom} to homology with rational and integer coefficients which will the subject of a forthcoming paper.\footnote{The generalization of Theorem \ref{teo:hom} mentioned above has already appeared in print \cite{pacoyo1}. In spite of this we have decided to still include Theorem \ref{teo:hom} here in order to make the paper as self contained as possible, since this result is used later on in the proof of Theorems \ref{teo:2mfds} and \ref{teo:3var}.}

\section{Attractors in surfaces \label{sec:2mfds}}

Now we assume that $K$ is a connected attractor in a connected $2$--manifold $M$. The key result is a characterization of what compacta $K \subseteq M$ have polyhedral shape:

\begin{lemma} \label{lem:2mfds} Let $K$ be a continuum in a surface $M$. If $\check{H}_1(K;\mathbb{Z}_2)$ is finitely ge{\-}nerated, then $K$ has the shape of a wedge sum of $r = {\rm rk}\ \check{H}^1(K;\mathbb{Z}_2)$ circumferences. In particular, $K$ has the shape of a polyhedron.
\end{lemma}

When $M = \mathbb{R}^2$ we can appeal to a result of Borsuk \cite[Theorem 7.1, p. 221]{borsukshape1} which states that two continua $K, K' \subseteq \mathbb{R}^2$ have the same shape if, and only if, they have the same first Betti numbers; that is, ${\rm rk}\ \check{H}^1(K;\mathbb{Z}_2) = {\rm rk}\ \check{H}^1(K';\mathbb{Z}_2)$. Under the hypotheses and notation of Lemma \ref{lem:2mfds}, let $K'$ be a wedge sum of $r$ polyhedral circumferences in the plane. Evidently $K$ and $K'$ have the same first Betti number, so they have the same shape. As a consequence, $K$ has the shape of a polyhedron (namely $K'$). In the general case (when $M$ is an arbitrary surface) it is possible to give an elementary proof of the lemma along the lines of Borsuk's result, but the argument is rather lengthy. Thus we will give an alternative short proof, at the cost of losing much of its intuitive content.

\begin{proof}[Proof of Lemma \ref{lem:2mfds}] We assume $M$ connected without loss of generality. If $K = M$ there is nothing to prove, because then $K$ is a compact surface, hence a polyhedron. So we consider the case when $K$ is a proper subset of $M$. A result of Nowak \cite[Corollary 3.9, p. 220]{nowak1} implies that ${\rm sd}\ K \leq 1$ and a result of Krasinkiewicz \cite[Theorem 7.2, p. 162]{krasinkiewicz1} guarantees that $K$ is pointed movable. Trybulec \cite[Remark 2, p. 203]{mardesic1} has proved that such a $K$ has the shape of a bouquet of a finite number of circumferences (possibly zero, i. e. a single point) or of the Hawaiian earring. But $\check{H}_1(K;\mathbb{Z}_2)$ is finitely generated by assumption, which rules out the second possibility.
\end{proof}

\begin{proof}[Proof of Theorem \ref{teo:2mfds}] By Theorem \ref{teo:hom} the group $\check{H}_1(K;\mathbb{Z}_2)$ is finitely generated. Then Lemma \ref{lem:2mfds} implies that $K$ has the shape of a polyhedron, so Theorem \ref{teo:main} applies.
\end{proof}

\section{Attractors in $\mathbb{R}^3$ \label{sec:3mfds}}

For this section we fix the following notation:
\medskip

{\bf Notation.} $f$ is a homeomorphism of $\mathbb{R}^3$ having a connected attractor $K$. $P \subseteq \mathbb{R}^3$ is a connected compact $3$--manifold (with boundary) contained in $\mathcal{A}(K)$ that satisfies $fP \subseteq {\rm int}\ P$. In particular $K \subseteq {\rm int}\ P$.
\medskip

When confronted with a particular $f$, the standard way to prove that it has an attractor is precisely by finding a trapping region $P$ which, in most cases, will satisfy the conditions enumerated above. However, let us show that (for the purposes of this paper) we can assume without loss of generality that a $P$ as described always exists. First observe that by Proposition \ref{prop:basica} there is no loss in generality in assuming that $K$ is connected. For each $p \in K$ pick a small closed cube centered at $p$ and contained in $\mathcal{A}(K)$. This family of cubes covers $K$ and, since $K$ is compact, we may discard all but finitely many of these cubes while still covering $K$. The union of the remaining cubes is a compact polyhedron $Q$ that contains $K$ and is contained in $\mathcal{A}(K)$. Finally, let $P$ be a regular neighbourhood of $Q$ (in the sense of piecewise linear topology; see for instance \cite[Chapter 3, pp. 31ff.]{rourkesanderson1}) so small that it is still contained in $\mathcal{A}(K)$. Then $P$ is a compact, connected $3$--manifold \cite[Proposition 3.10, p. 34]{rourkesanderson1} that is a neighbourhood of $K$ contained in $\mathcal{A}(K)$. In particular, it is attracted by $K$ and so there exists $N \geq 1$ such that $f^N P \subseteq {\rm int}\ P$. Replacing $f$ with $f^N$ (something that we are entitled to do by Proposition \ref{prop:potencia}) we may simply assume that $fP \subseteq {\rm int}\ P$.
\medskip

Choose a basepoint $* \in K$ and consider the sequence \[\mathcal{H} :  \xymatrix{\ldots & \pi_1(f^{-1}P,*) \ar[l]_-{\varphi_{-1}} & \pi_1(P,*) \ar[l]_-{\varphi_0} & \pi_1(fP,*) \ar[l]_-{\varphi_1} & \ldots \ar[l]_-{\varphi_2}} \] where the $\varphi_n$ denote the inclusion induced homomorphisms. As usual, for each $n \geq 0$ we denote $F_n := {\rm im}\ (\varphi_1  \ldots  \varphi_n)$, which in this case is nothing but the image of the inclusion induced homomorphism $\pi_1(f^nP,*) \longrightarrow \pi_1(P,*)$. We consider the sequence \[(S) : \pi_1(P,*) = F_0 \geq F_1 \geq F_2 \geq \ldots\]

\begin{proposition} \label{prop:rigid} $\mathcal{H}$ is rigid.
\end{proposition}
\begin{proof} This is essentially contained in the proof of the surjectivity of $i_*$ in Theorem \ref{teo:main}. Let $n,m \in \mathbb{Z}$, $m \geq 0$. Choose a path $\gamma$ contained in $f^{n+m}P$ joining $*$ and $f^n(*)$. Denote $h_{\gamma}$ the basepoint change isomorphism $h_{\gamma}(\alpha) := \gamma \cdot \alpha \cdot \gamma^{-1}$. The following diagram is conmutative

\[\xymatrix@=15mm{\pi_1(P,*) \ar[d]^{h_{\gamma}f^n_*}_{\cong} & \pi_1(fP,*) \ar[l]_{\varphi_1} \ar[d]^{h_{\gamma}f^n_*}_{\cong} & \ldots \ar[l]_{\varphi_2} \ar@{}[d]|{\vdots} & \pi_1(f^mP,*)  \ar[l]_{\varphi_m} \ar[d]_{\cong}^{h_{\gamma}f^n_*} \\ \pi_1(f^nP,*) & \pi_1(f^{n+1}P,*) \ar[l]^{\varphi_{n+1}} & \ldots \ar[l]^{\varphi_{n+2}} & \pi_1(f^{n+m}P,*) \ar[l]^{\varphi_{n+m}}}\]

The vertical arrows are isomorphisms because the restrictions $f^n|_P : P \longrightarrow f^nP$, $f^n|_{fP} : fP \longrightarrow f^{n+1}P$, $\ldots$ are homeomorphisms. Thus $\mathcal{H}$ is rigid.
\end{proof}

Our main tool here is a powerful theorem of Kadlof \cite[Theorem 4.3, p. 180]{kadlof1}:

\smallskip
{\it Theorem.} Let $(K,*)$ be a continuum in $\mathbb{R}^3$. Assume that ${\rm pro-}\pi_1(K,*)$ has the Mittag--Leffler property, $\check{\pi}_1(K,*)$ is countable and $\check{H}_2(K;\mathbb{Z}_2)$ is finitely generated. Then $K$ has the shape of a compact polyhedron.
\smallskip

Kadlof's result is formulated in terms of pointed $1$--movability, which as mentioned in the introduction is equivalent to ${\rm pro-}\pi_1(K,*)$ having the Mittag--Leffler property.

\begin{theorem} \label{teo:3var1} The inclusion $i : K \subseteq \mathcal{A}(K)$ is a shape equivalence if, and only if, there exists $n$ such that every loop based at $*$ and contained in $f^nP$ can be homotoped, within $P$ and keeping $*$ fixed, to a loop contained in $f^{n+1}P$. When this is the case, $K$ has the shape of a compact polyhedron.
\end{theorem}
\begin{proof} Notice that the condition mentioned in the theorem is the verbal expression of the equality $F_n = F_{n+1}$. Since $\mathcal{H}$ is rigid, this is equivalent to saying that $(S)$ stabilizes, and in turn to assuming that ${\rm pro-}\pi_1(K,*)$ has the Mittag--Leffler property, or that $K$ is pointed $1$--movable.

Suppose $i$ is a shape equivalence. Then $(K,*)$ has the shape of a polyhedron and therefore ${\rm pro-}\pi_1(K,*)$ has the Mittag--Leffler property.

Conversely, assume ${\rm pro-}\pi_1(K,*)$ has the Mittag--Leffler property. By Theorem \ref{teo:main} we only need to show that $K$ has the shape of a polyhedron, and we do this using the theorem of Kadlof quoted above. From Theorem \ref{teo:hom} we know that $\check{H}_2(K;\mathbb{Z}_2)$ is finitely generated, so it only remains to prove that $\check{\pi}_1(K,*)$ is countable. To this end we shall apply Theorem \ref{teo:algebra} to the sequence $\mathcal{H}$. Hypothesis ({\it a}\/) is satisfied because of Proposition \ref{prop:rigid}, whereas ({\it b}\/) holds by assumption. Let us consider hypothesis ({\it c}\/). The group $\pi_1(P,*)$ is finitely generated because it is the fundamental group of a compact $3$--manifold. It is also true that it is residually finite, although this is much more involved to prove: it follows from the geometrisation conjecture of Thurston \cite[Theorem 3.3, p. 364]{thurston1}. Hence the three hypotheses of Theorem \ref{teo:algebra} are satisfied, and therefore $\underline{\varphi} : \check{\pi}_1(K,*) \longrightarrow \pi_1(P,*)$ is injective. Since $\pi_1(P,*)$ is countable (because it is finitely generated), the same is true of $\check{\pi}_1(K,*)$. The result follows.
\end{proof}

\begin{proof}[Proof of Corollary \ref{cor:loccon}] A continuum having at most a countable number of path connected components is pointed $1$--movable \cite[Corollary 3.2, p. 152]{krasinkiewiczminc1} so ${\rm pro-}\pi_1(K,*)$ has the Mittag--Leffler property and $(S)$ stabilizes. Thus $F_n = F_{n+1}$ for some $n$ and Theorem \ref{teo:3var1} applies.
\end{proof}

Remark \ref{rem:fg} and the proof of Theorem \ref{teo:3var1} establish the following result:

\begin{remark} \label{rem:fg3var} If $i : K \subseteq \mathcal{A}(K)$ is a shape equivalence, then $\check{\pi}_1(K,*)$ is isomorphic to both a subgroup and a quotient of $\pi_1(P,*)$.
\end{remark}

Suppose for a moment that the basepoint $*$ were actually a fixed point for $f$, as considered in the introduction just before stating Theorem \ref{teo:3var}. Recall that we denoted $\Phi : \pi_1(P,*) \longrightarrow \pi_1(P,*)$ the homomorphism induced by $f|_P : P \longrightarrow P$ on fundamental groups, and said that $\Phi$ \emph{stabilizes} if the sequence \[\pi_1(P,*) \geq {\rm im}\ \Phi \geq {\rm im}\ \Phi^2 \geq \ldots\] stabilizes in the sense that there exists $n$ such that ${\rm im}\ \Phi^n = {\rm im}\ \Phi^{n+1}$ for some $n$. Under these hypotheses the proof of Theorem \ref{teo:3var} is trivial:

\begin{proof}[Proof of Theorem \ref{teo:3var} (weak version)] Let us write $f^n\|_P : P \longrightarrow f^nP$ for the homeomorphism that results from restricting the domain of $f^n$ to $P$ \emph{and its range} (this is the difference with $f|_P$) to $f^nP$. Then $(f|_P)^n$ equals the composition \[\xymatrix@=15mm{P \ar[r]^{f^n\|_P}_{\cong} & f^nP \ar[r] & P}\] where the unlabeled arrow denotes, as usual, an inclusion. Consequently $\Phi^n$ factors as \[\xymatrix@=15mm{\pi_1(P,*) \ar[r]^{(f^n\|_P)_*}_{\cong} & \pi_1(f^nP,*) \ar[r] & \pi_1(P,*)}\] so, the first arrow being surjective, ${\rm im}\ \Phi^n$ is precisely the image of the second arrow, which is $F_n$. Thus $(S)$ stabilizes if and only if $\Phi$ stabilizes. With Theorem \ref{teo:3var1}, this proves Theorem \ref{teo:3var}.
\end{proof}

Let us consider the general case where $*$ is just a point in $P$ but not necessarily in $K$ or a fixed point for $f$. Now $\Phi$ cannot be defined as readily as before, because $(f|_P)_*$ is not an endomorphism of $\pi_1(P,*)$, but we can still define a map $\Phi : \pi_1(P,*) \longrightarrow \pi_1(P,*)$ by choosing a path $\gamma \subseteq P$ from $*$ to $f*$ and letting $\Phi$ be the composition \[\xymatrix@=15mm{\pi_1(P,*) \ar[r]^{(f|_P)_*} & \pi_1(P,f*) \ar[r]^{h_{\gamma}}_{\cong} & \pi_1(P,*)}\] where $h_{\gamma}$ denotes the basepoint change isomorphism $h_{\gamma}([\alpha]) = [\gamma \cdot \alpha \cdot \gamma^{-1}]$. The equality $F_n = {\rm im}\ \Phi^n$ is no longer true, so the argument used above to prove Theorem \ref{teo:3var} is not valid. However, the content of the theorem itself (with this new definition of $\Phi$) still holds, and this is what we aim to prove now. We need a couple of auxiliary results.
\medskip

Let $f' : \mathbb{R}^3 \longrightarrow \mathbb{R}^3$ be another homeomorphism for which $f'P \subseteq {\rm int}\ P$ too, so that $f'$ has an attractor $K' \subseteq {\rm int}\ P$.

\begin{proposition} \label{prop:inv} Suppose $f|_P, f'|_P : P \longrightarrow P$ are homotopic. Then there is a shape equivalence $k : K \longrightarrow K'$ such that the diagram \[\xymatrix{ & K \ar[dd]^{\cong}_{k} \ar[dl] \\ P & \\ & K' \ar[ul]}\] commutes in the shape category. The unlabeled arrows denote shape morphisms induced by inclusions.
\end{proposition}
\begin{proof} It is very easy to check that the inverse sequences \[\xymatrix@=15mm{P & P \ar[l]_{f|_P} & P \ar[l]_{f|_P} & \ldots \ar[l]_{f|_P}}\] and \[\xymatrix@=15mm{P & P \ar[l]_{f'|_P} & P \ar[l]_{f'|_P} & \ldots \ar[l]_{f'|_P}}\] have $K$ and $K'$ as their inverse limits, respectively. Therefore they are ANR--expansions of $K$ and $K'$, because $P$ is an ANR. Since $f|_P$ and $f'|_P$ are homotopic, the following diagram commutes in the homotopy category: \[\xymatrix@=15mm{P \ar@<0.5ex>[d]^{{\rm id}_P} & P \ar[l]_{f|_P} \ar@<0.5ex>[d]^{{\rm id}_P} & P \ar[l]_{f|_P} \ar@<0.5ex>[d]^{{\rm id}_P} & \ldots \ar[l]_{f|_P} \\ P \ar@<0.5ex>[u]^{{\rm id}_P} & P \ar[l]_{f'|_P} \ar@<0.5ex>[u]^{{\rm id}_P}  & P \ar[l]_{f'|_P} \ar@<0.5ex>[u]^{{\rm id}_P}  & \ldots \ar[l]_{f'|_P}}\]

The sequence of downward arrows $\{{\rm id}_P\}$ defines a shape morphism $k : K \longrightarrow K'$ whose shape inverse is obviously given by the sequence of upward arrows $\{{\rm id}_P\}$, so $k$ is a shape equivalence. It is also clear that $k$ commutes with the inclusions of $K$ and $K'$ in $P$.
\end{proof}

For technical reasons we shall need to assume that $\gamma$ is a polygonal path contained in ${\rm int}\ P$. There is clearly no loss of generality in doing this, since $\gamma$ can be replaced by a polygonal approximation $\gamma'$ contained in ${\rm int}\ P$ which (if chosen sufficiently close to $\gamma$) gives rise to the same map $h_{\gamma'} = h_{\gamma}$, thus leaving $\Phi$ unchanged. 

\begin{lemma} \label{lem:isotopia} Let $\gamma : [0,1] \longrightarrow \mathbb{R}^3$ be a polygonal path starting at $p$ and finishing at $q$. Let $W$ be an open neighbourhood of ${\rm im}\ \gamma$. Then there exists an isotopy $(h_t)_{t \in [0,1]} : \mathbb{R}^3 \longrightarrow \mathbb{R}^3$ such that ({\it i}\/) $h_0 = {\rm id}$, ({\it ii}\/) each $h_t$ is the identity outside $W$, ({\it iii}\/) $h_t(q) = \gamma(1-t)$ for every $t \in [0,1]$.
\end{lemma}

The proof of Lemma \ref{lem:isotopia} is easy, so we omit it. In order to understand the behaviour of $h_t$ at the homotopy level, let $\alpha$ be a loop in $W$ based at $q$. Each $h_t \alpha$ is a loop in $W$ based at $h_t(q) = \gamma(1-t)$, and we can picture $\alpha$ as sliding backwards along $\gamma$ until it reaches its final position $h_1 \alpha$, based at $p$. See Figure \ref{fig:isotopy}.

\begin{figure}[h]
\begin{pspicture}(-0.2,-0.2)(8.4,4.5)
	\rput[t](0,2){$p$} \rput[tr](2.1,1.5){$\gamma(1-t)$} \rput[t](5.2,-0.1){$q$}
	\rput[bl](1.8,4.2){$h_1 \alpha$} \rput[bl](3.8,3.8){$h_t\alpha$} \rput[bl](7.8,2.2){$\alpha$}
	\rput[bl](0,0){\scalebox{0.9}{\includegraphics{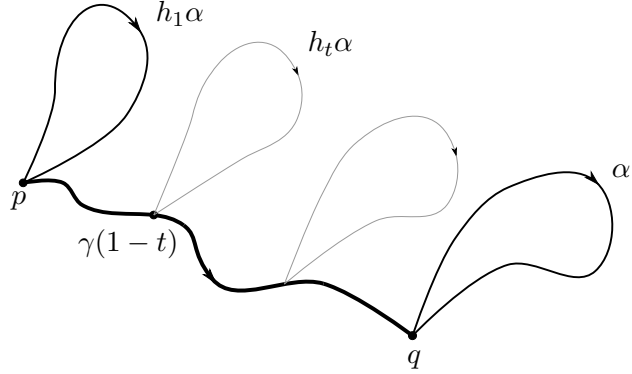}}}
\end{pspicture}
\caption{The action of $h_t$ on a loop $\alpha$ \label{fig:isotopy}}
\end{figure}

It is clear from Figure \ref{fig:isotopy} that $\alpha \simeq \gamma^{-1} \cdot h_1 \alpha \cdot \gamma$ (rel. $q$) in $W$. Therefore the map $(h_1)_* : \pi_1(W,q) \longrightarrow \pi_1(W,p)$ is precisely the basepoint change isomorphism $h_{\gamma}$.

\begin{proof}[Proof of Theorem \ref{teo:3var} (strong version)] Let $\gamma$ be the path connecting $*$ with $f*$ that was used in defining $\Phi$. As mentioned earlier, we may assume that $\gamma$ is a polygonal path in ${\rm int}\ P$. Choose an open neighbourhood $W$ of $\gamma$ such that $\overline{W} \subseteq {\rm int}\ P$. By Lemma \ref{lem:isotopia} there exists an isotopy $(h_t)_{t \in [0,1]}$ such that $h_0$ is the identity, each $h_t$ is the identity outside $W$ (hence also outside ${\rm int}\ P$) and $h_1(f*) = *$. Define $f_t := h_t  f$ for each $t \in [0,1]$ and denote $f' := f_1$ for notational consistency. Notice that $f'$ has $P$ as a trapping region, so it has an attractor $K' \subseteq {\rm int}\ P$. Also, since $f$ and $f'$ agree outside $P$, clearly $\mathcal{A}(K) = \mathcal{A}(K')$. By construction $f|_P : P \longrightarrow P$ and $f'|_P : P \longrightarrow P$ are homotopic (actually isotopic), so by Proposition \ref{prop:inv} there is a shape equivalence $k : K \longrightarrow K'$ such that the diagram \[\xymatrix{& & K \ar[dd]^{\cong}_{k} \ar[dl] \\ \mathcal{A}(K) = \mathcal{A}(K') & P \ar[l] & \\ & & K' \ar[ul]}\] commutes in the shape category. It follows that $i : K \subseteq \mathcal{A}(K)$ is a shape equivalence if, and only if, $i' : K' \subseteq \mathcal{A}(K')$ is a shape equivalence.

Observe that $f'* = h_1f* = *$ by construction, so that $*$ is a fixed point for $f'$ (in particular, $* \in K'$). Denote $\Phi' : \pi_1(P,*) \longrightarrow \pi_1(P,*)$ the homomorphism induced by $f'|_P$ on fundamental groups. Applying the weak version of Theorem \ref{teo:3var} proved earlier to $f'$ we see that $i' : K' \subseteq \mathcal{A}(K')$ is a shape equivalence if, and only if, $\Phi'$ stabilizes. However, since $f'|_P = h_1 f|_P$ by construction, $\Phi' = (h_1)_* (f|_P)_* = h_{\gamma} (f|_P)_* = \Phi$. Thus $\Phi'$ is stabilizes if, and only if, $\Phi$ stabilizes. This concludes the proof.
\end{proof}

\subsection*{Trapping regions with free fundamental group} Theorem \ref{teo:3var} leads naturally to the following very general algebraic problem: given a group $G$ and an endomorphism $\alpha : G \longrightarrow G$, is there an effective way to decide whether $\alpha$ stabilizes? (in our case $G = \pi_1(P,*)$, where $P$ is a trapping region for $f$, and $\alpha = \Phi$ which is essentially $(f|_P)_*$). Owing to the discussion before Theorem \ref{teo:3var} we shall understand that ``given $G$ and $\alpha$'' means that the following data is available to us:
\begin{itemize}
	\item a finite presentation $G = \langle g_1,g_2, \ldots, g_k : r_1, r_2, \ldots r_{\ell} \rangle$,
	\item the elements $\alpha(g_1), \ldots, \alpha(g_k)$ written as words in the $g_i$.
\end{itemize}

In this section we are interested in the case when $P$ is a handlebody or, more generally, $G$ is a finitely generated free group. We want to show that it is possible to perform a computation in terms of $G$ and $\alpha$ that will decide whether the inclusion $i$ is a shape equivalence or not.

For brevity we shall denote $F_n = {\rm im}\ \alpha^n$, as usual. Each $F_n$ is a free group (being a subgroup of a free group) and finitely generated (being an image of the finitely generated group $G$). Let $r_n := {\rm rk}\ F_n$. Since $F_{n+1} = \alpha(F_n)$, it follows that $r_{n+1} \leq r_n$, so the sequence $\{r_n\}$ is decreasing.

\begin{remark} \label{rem:libre} Let $\alpha : G \longrightarrow G$ be an endomorphism of a free group $G$ of finite rank. Then \[\alpha \text{ is injective } \Leftrightarrow {\rm rk}\ G = {\rm rk}\ {\rm im}\ \alpha.\]
\end{remark}
\begin{proof} ($\Rightarrow$) is trivial. Conversely ($\Leftarrow$), assume ${\rm rk}\ G = {\rm rk}\ {\rm im}\ \alpha$. Since free groups are uniquely determined by their rank, there exists an isomorphism $\theta : {\rm im}\ \alpha \longrightarrow G$, and then $\theta \alpha : G \longrightarrow G$ is an epimorphism. Finitely generated free groups are Hopfian \cite[Proposition 3.5, p. 14]{lyndonschupp1}, so $\theta \alpha$ is an isomorphism, and consequently $\alpha$ is injective.
\end{proof}

This remark has the following two consequences which are of interest to us:

\begin{proposition} \label{prop:estabiliza} The sequence $\{r_n\}$ is strictly decreasing until the first equality appears at some stage $m$, and from then on it becomes constant: \[r_0 = {\rm rk}\ G > r_1 > \ldots > r_m = r_{m+1} = r_{m+2} = \ldots\] where $m \leq {\rm rk}\ G$.
\end{proposition}
\begin{proof} Assume that $r_n = r_{n+1}$ for some $n$. Applying Remark \ref{rem:libre} to the map $\alpha|_{F_n}$ we see that it has to be injective, and consequently the same is true of $\alpha|_{F_{n+1}}$. Then again by Remark \ref{rem:libre} we have $r_{n+1} = r_{n+2}$, and so on. Thus $\{r_n\}$ is strictly decreasing until it becomes constant. The inequality $m \leq {\rm rk}\ G$ is trivial.
\end{proof}

\begin{proposition} \label{prop:libre1} Choose any $n$ such that $r_n = r_{n+1}$ (for instance, let $n = {\rm rk}\ G$). Then $\alpha$ stabilizes if, and only if, \[(*) \ \ \alpha^n(g_i) \in {\rm gp}\ \{\alpha^{n+1}(g_1),\ldots,\alpha^{n+1}(g_k)\}\] for every generator $g_i$ of $G$.
\end{proposition}
\begin{proof} If $(*)$ holds then $F_n = F_{n+1}$ so $\alpha$ stabilizes. Conversely, assume $\alpha$ stabilizes so $F_m = F_{m+1}$ for sufficiently large $m$, which we may as well choose to satisfy $m \geq n$. Since $r_n = r_{n+1}$, it follows from Remark \ref{rem:libre} that $\alpha|_{F_n}$ is injective, so the same is true of all of its powers. Thus \[\emptyset = F_m \backslash F_{m+1} = \alpha^{m-n}(F_n) \backslash \alpha^{m-n}(F_{n+1}) = \alpha^{m-n}(F_n \backslash F_{n+1}),\] so $F_n = F_{n+1}$ and $(*)$ holds.
\end{proof}

The interesting point about Proposition \ref{prop:libre1} is that $(*)$ can be algorithmically checked in free groups. We present an approach due to Stallings \cite{stallings1} (see also \cite{kapovich1}) but, rather than explaining it theoretically, we consider a specific example.

\begin{example} \label{ejem:stallings} Suppose $f$ has a trapping region $P$ that is a handlebody with two handles, so that $\pi_1(P,*)$ is the free group on two generators $\langle x,y : \emptyset \rangle$. Assume that inspection of $f$ shows that $\Phi$ acts on the generators as follows: $\Phi(x) = xy^2$ and $\Phi(y) = x^2y^3$. Then $i : K \subseteq \mathcal{A}(K)$ is not a shape equivalence.
\end{example}
\begin{proof} Abelianizing we get a homomorphism $\bar{\Phi} : \mathbb{Z}^2 \longrightarrow \mathbb{Z}^2$ whose matrix is \[\begin{pmatrix} 1 & 2 \\ 2 & 3 \end{pmatrix}\] which is invertible, so $\bar{\Phi}$ is an isomorphism. Thus $r_0 = r_1 = 2$ and according to Proposition \ref{prop:libre1} we only need to check whether $x$ and $y$ belong to ${\rm im}\ \Phi = {\rm gp}\ \{xy^2,x^2y^3\}$ (that is, whether $\Phi$ is surjective).

Let $\Gamma$ be the graph with a single vertex and two loops labeled $x$ and $y$. Its fundamental group is $\pi_1(P,*)$, the free group in two generators. We may represent ${\rm im}\ \Phi$ by the graph $\Sigma$ shown in Figure \ref{fig:groups}. $\Sigma$ has two loops based at vertex $v$, each corresponding to one of the generators of ${\rm im}\ \Phi$. For instance, the left loop comprises three edges labeled $x$, $y$, $y$ (to unclutter the drawing we represent $x$ by a single arrow and $y$ by a double arrow). The right loop comprises five edges labeled $x$, $x$, $y$, $y$, $y$. We say that this graph represents ${\rm im}\ \Phi$ in the sense that the natural map $p : \Sigma \longrightarrow \Gamma$ sending each edge onto the corresponding edge of $\Gamma$ induces a homomorphism between $\pi_1(\Sigma,v)$ and $\pi_1(\Gamma)$ whose image is precisely ${\rm im}\ \Phi$.

Now consider the graph $\Sigma'$ obtained by collapsing the two $x$ edges that exit $v$ into a single one and, similarly, the two $y$ edges that enter $v$ into a single one (refer to Figure \ref{fig:groups}). The map $p : \Sigma \longrightarrow \Gamma$ obviously gives rise to a map $p' : \Sigma' \longrightarrow \Gamma$ such that ${\rm im}\ p'_*$ is still ${\rm im}\ \Phi$. Finally, $\Sigma''$ is obtained from $\Sigma'$ by collapsing the two $y$ edges that enter $w$ into a single one. Again, $\Sigma''$ still represents ${\rm im}\ \Phi$, but it has an important advantage over $\Sigma$: now each vertex has at most an incoming $x$ edge and an outcoming $x$ edge, and similarly for the $y$ edges. Thus the map $p'' : \Sigma'' \longrightarrow \Gamma$ has the unique path lifting property, which neither $p'$ nor $p$ have. In the terminology of Stallings' paper \cite[pp. 554 and 555]{stallings1} each graph has been obtained \emph{folding} the previous one, and the map $p''$ is an \emph{immersion}.

\begin{figure}[h]
\begin{pspicture}(-0.2,0)(12.4,2.4)
	\rput[bl](0,0){\scalebox{0.8}{\includegraphics{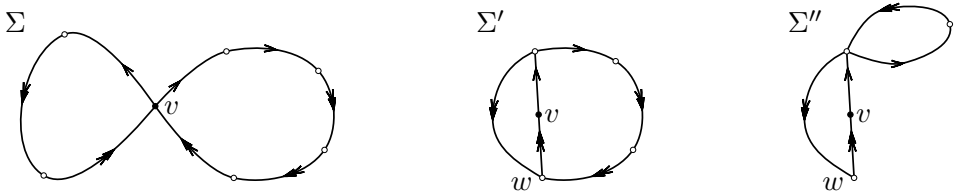}}}
	\rput[bl](-0.2,2){$\Sigma$} \rput[bl](6,2){$\Sigma'$} \rput[bl](10.1,2){$\Sigma''$}
	\rput(2,1){$v$} \rput(7,0.9){$v$} \rput(6.6,0){$w$} \rput(11.1,0.9){$v$} \rput(10.7,0){$w$}
\end{pspicture}
\caption{Checking whether $\Phi$ is surjective \label{fig:groups}}
\end{figure}

From the last graph we readily see that $x$ does not belong to ${\rm im}\ \Phi$ (neither does $y$) because its lift starting from $v$ is not a closed path. Thus $\Phi$ is not surjective.
\end{proof}

\begin{theorem} \label{teo:free} Let $f$ be a homeomorphism of $\mathbb{R}^3$ having a connected trapping region $P$ with free fundamental group. Let $K \subseteq {\rm int}\ P$ be the attractor inside $P$. Then the inclusion $i : K \subseteq \mathcal{A}(K)$ is a shape equivalence if and only if $K$ has the shape of a finite wedge sum of circumferences $\mathbb{S}^1$ and $2$--spheres $\mathbb{S}^2$.
\end{theorem}

The proof depends on the following result of Kadlof: if $K, K' \subseteq \mathbb{R}^3$ are two continua with polyhedral shape such that $\check{\pi}_1(K) \cong \check{\pi}_1(K')$ and $\check{H}_2(K;\mathbb{Z}_2) \cong \check{H}_2(K';\mathbb{Z}_2)$, then they have the same shape \cite[Corollary 3.4, p. 178]{kadlof1}. In particular, if $\check{\pi}_1(K)$ is the free group of rank $r$ and $\check{H}_2(K;\mathbb{Z}_2)$ has rank $s$ (so it is isomorphic to $\mathbb{Z}_2^s$) then $K$ has the shape of a wedge sum of $r$ circumferences and $s$ $2$--spheres $\mathbb{S}^2$.

\begin{proof}[Proof of Theorem \ref{teo:free}] If $K$ has the shape of a finite wedge sum of circumferences and $2$--spheres, then it has polyhedral shape and so $i$ is a shape equivalence by Theorem \ref{teo:main}. Conversely, assume $i$ is a shape equivalence. Since $P$ is connected, $K$ is connected too, so we will omit basepoints from the notation. By Remark \ref{rem:fg3var} $\check{\pi}_1(K,*)$ is isomorphic to both a subgroup and a quotient of $\pi_1(P,*)$, which is a finitely generated (being $P$ a compact $3$--manifold) free group (by hypothesis). Thus $\check{\pi}_1(K,*)$ is also a finitely generated free group, say of rank $r$. Also, Theorem \ref{teo:hom} shows that $\check{H}_2(K;\mathbb{Z}_2)$ is finitely generated, so it is isomorphic to some $\mathbb{Z}_2^s$. Then by Kadlof's theorem $K$ has the shape of a finite wedge sum of circumferences and $2$--spheres.
\end{proof}

The same method of proof establishes the following:

\begin{corollary} If $P$ is a handlebody, then $i$ is a shape equivalence if and only if $K$ has the shape of a finite bouquet of circumferences.
\end{corollary}

\section{Appendix: the general case of Theorem \ref{teo:hom}} \label{app:trick}

To prove Theorem \ref{teo:hom} in the general case we shall reduce it to the weak case already considered, where assumption $(A)$ is satisfied. Recall that $K$ is an attractor for a homeomorphism $f$ of a locally compact, metrizable ANR $M$. Denote $Q$ the Hilbert cube. Clearly $f \times {\rm id}_Q : M \times Q \longrightarrow M \times Q$ is a homeomorphism having $K \times Q$ as an attractor, and further \[\hat{f} := f \times {\rm id}_Q \times (t \mapsto t^2) : M \times Q \times [0,1) \longrightarrow M \times Q \times [0,1)\] is a homeomorphism having $\hat{K} := K \times Q \times 0$ as an attractor with basin of attraction $\mathcal{A}(\hat{K}) = \mathcal{A}(K) \times Q \times [0,1)$.

\begin{proposition} \label{prop:hom} $\hat{K}$ has a neighbourhood $P \subseteq \mathcal{A}(\hat{K})$ that is a compact ANR.
\end{proposition}
\begin{proof} By the ANR theorem of Edwards \cite[Chapter XIV]{chapman2} $\mathcal{A}(K) \times Q$ is a $Q$--manifold, and so $\mathcal{A}(\hat{K}) = \mathcal{A}(K) \times Q \times [0,1)$ is triangulable \cite[Corollary 1, p. 330]{chapman1}. This means that there exists a countable, locally finite simplicial complex $S$ such that $\mathcal{A}(\hat{K}) \cong |S| \times Q$ via some homeomorphism $h$. In order to apply Edwards' theorem we have to check that $\mathcal{A}(K)$ is separable, but this is easy: pick a compact neighbourhood $A$ of $K$; then $A$ is compact and metrizable, hence separable, and $\mathcal{A}(K) = \bigcup_{n \leq 0} f^nA$ is a countable union of separable spaces, therefore also separable. Since $hK$ is a compact subset of $|S| \times Q$, its projection ${\rm pr}_{|S|}hK$ onto the $|S|$ factor is a compact set so there exists a finite subcomplex of $S$ such that the union of its simplices is a (polyhedral) compact neighbourhood $N$ of ${\rm pr}_{|S|}hK$ in $|S|$. Then $N \times Q$ is a compact ANR neighbourhood of $hK$ in $|S| \times Q$, and its preimage $P$ under $h$ is a compact ANR neighbourhood of $\hat{K}$ in $\mathcal{A}(\hat{K})$.
\end{proof}

Proposition \ref{prop:hom} entitles us to apply the weak version of Theorem \ref{teo:hom} to $\hat{K}$ to conclude that the inclusion $j : \hat{K} \subseteq \mathcal{A}(\hat{K})$ induces isomorphisms on \v{C}ech homology with $\mathbb{Z}_2$ coefficients and $\check{H}_*(\hat{K};\mathbb{Z}_2)$ is finitely generated. Consider $K$ and $\mathcal{A}(K)$ embedded in $\hat{K}$ and $\mathcal{A}(\hat{K})$ respectively via $e(p) := (p,0,0)$: then in the commutative diagram \[\xymatrix@=15mm{\hat{K} \ar[r]^j & \mathcal{A}(\hat{K}) \\ K \ar[u]^{e|_K} \ar[r]_i & \mathcal{A}(K) \ar[u]_e}\] both vertical arrows are homotopy equivalences and therefore induce isomorphisms in \v{C}ech homology. The general case of Theorem \ref{teo:hom} follows.

\section{Appendix: some background on shape theory \label{ap:shape}}

Denote ${\bf HMet}$ the category of metric spaces and homotopy classes of continuous maps between them. The objects of the shape category ${\bf SMet}$ are the same as those in ${\bf HMet}$, just metric spaces. Its morphisms are, however, a little more complicated. Every continuous map $f : X \longrightarrow Y$ between two metric spaces induces a shape morphism which depends only on the homotopy class of $f$ (we denote it again $f : X \longrightarrow Y$), but in general there may exist shape morphisms $u : X \longrightarrow Y$ that are not induced by any continuous map from $X$ to $Y$. Thus ${\bf SMet}$ contains representatives of every morphism in ${\bf HMet}$ plus some extra ones that account for its added flexibility. However, when $Y$ is a polyhedron (or more generally an ANR) every shape morphism $u : X \longrightarrow Y$ is induced by a continuous map $u : X \longrightarrow Y$, which heuristically means that shape theory and homotopy theory agree on spaces that are locally well behaved, such as polyhedra or ANRs.

Of course, shape morphisms can be composed. A shape morphism $u : X \longrightarrow Y$ is a \emph{shape equivalence} if it has a \emph{shape inverse}, which is a shape morphism $v : Y \longrightarrow X$ such that $vu = {\rm id}_X$ and $uv = {\rm id}_Y$. We then say that $X$ and $Y$ have \emph{the same shape}. If $f : X \longrightarrow Y$ is a continuous map, it has ``more chances'' of having a shape inverse than a homotopy inverse, because as we said earlier there might be shape morphisms from $Y$ to $X$ that are not induced by any continuous map. Thus $f$ can be a shape equivalence in spite of not being a homotopy equivalence, as mentioned after Theorem \ref{teo:flujos}.

There are several ways to define the shape category but, since all of them are quite lengthy, we refer the reader to the book by Marde{\v{s}}i{\'{c}} and Segal \cite{mardesic1}, chapters VII to IX of the monograph by Borsuk \cite{borsukshape1} or even his foundational paper \cite{borsukshape2}. A survey by Marde{\v{s}}i{\'{c}} \cite{mardesic3} is also worth reading. We can, however, particularize to the situation of interest in this paper and state what it means for the inclusion $i : K \subseteq \mathcal{A}(K)$ to be a shape equivalence. For comparison purposes, consider first what it would mean to say that $i$ is a homotopy equivalence. This would mean that it has an inverse $j : \mathcal{A}(K) \longrightarrow K$ in the homotopy category; that is, $j$ is a continuous map such that $ji \simeq {\rm id}_K$ and $ij \simeq {\rm id}_{\mathcal{A}(K)}$. Now, in the shape category an inverse for $i$ is no longer given by a single continuous map $j : \mathcal{A}(K) \longrightarrow K$ but by a sequence of continuous maps $j_n : \mathcal{A}(K) \longrightarrow V_n$ (notice the change of target space). Here the $V_n$ are a decreasing sequence of open neighbourhoods of $K$ in $\mathcal{A}(K)$ such that $\bigcap_{n \geq 1} V_n = K$, and the maps $j_n$ satisfy the following conditions for each $n$:

\begin{itemize}
	\item[({\it a}\/)] $j_n \simeq j_{n+1}$ in $V_n$,
	\item[({\it b}\/)] $j_n i \simeq {\rm id}_K$ in $V_n$,
	\item[({\it c}\/)] $j_n \simeq {\rm id}_{\mathcal{A}(K)}$ in $\mathcal{A}(K)$.
\end{itemize}

Heuristically, ({\it a}\/) guarantees that there exists a sort of ``homotopy limit'' for the sequence $j_n$, which is precisely what shape theory formalizes as a shape morphism from $\mathcal{A}(K)$ to $K$. Condition ({\it b}\/) would then be the analogous to $ji \simeq {\rm id}_K$ whereas ({\it c}\/) corresponds to $ij \simeq {\rm id}_{\mathcal{A}(K)}$. It is a basic result in shape theory that the arbitrarity in the choice of the sequence of neighbourhoods $V_n$ is inconsequential.
\medskip

There follows a brief enumeration of some of the concepts and results in shape theory that we use, without further explanation, in this paper. This is just intended as a convenient reminder; we also provide suitable bibliographic references for key results.
\medskip

{\it\underline{Absolute neighbourhood retracts}}. Polyhedra are very useful for explicit geometric constructions because they are built upon the notion of simplicial complex, but they are uncomfortable to deal with from the point of view of general topology. Borsuk abstracted the most useful properties of polyhedra and introduced the notion of absolute neighbourhood retracts, which is purely topological but still very useful in geometric topology. We shall only consider locally compact, metrizable ANRs (for the class of metrizable spaces). There is a dual notion, that of absolute neighbourhood extensors (again, for metrizable spaces) which actually coincides with that of ANRs.

Rather than recalling the definition of ANR, which the reader can find elsewhere \cite{borsukretractos,hu}, we shall just mention some of their properties. A finite dimensional space is an ANR if, and only if, it is locally contractible \cite[Theorem 7.1, p. 168]{hu}. In particular, every manifold and every polyhedron is an ANR. Conversely, every ANR has the homotopy type of a polyhedron \cite[Theorem 5, p. 317]{mardesic1} so from the point of view of homotopy theory, polyhedra and ANRs are interchangeable. An open subset of an ANR is an ANR too \cite[Proposition 7.9, p. 97]{hu}. 

\medskip

{\it\underline{ANR expansions}}. Suppose $X$ is a compact subset of a metric space $M$ and $P_k$ is a decreasing sequence of neighbourhoods of $X$ that are ANRs and whose intersection is precisely $X$. A familiar example is when $M$ is a manifold and the $P_k$ are open neighbourhoods of $X$. We may describe this situation in a fancier language by saying that $X$ is the inverse limit of the inverse sequence \[P_1 \longleftarrow P_2 \longleftarrow P_3 \longleftarrow \ldots\] where each bonding map (each arrow) denotes the inclusion of $P_{k+1}$ into $P_k$. More generally, any inverse sequence of ANRs whose inverse limit is $X$ is called an \emph{ANR expansion} of $X$, and shape morphisms between two compact sets $X$ and $Y$ can be defined in terms of suitable sequences of maps between any two ANR expansions of $X$ and $Y$. ANR expansions can also be defined for noncompact (metric) spaces $X$, although not in such simple terms.

\medskip

{\it\underline{Shape invariants}}. Let \[\xymatrix{P_1 & P_2 \ar[l]_-{p_2} & P_3 \ar[l]_-{p_3} & \ldots \ar[l]}\] be an ANR expansion of a (not necessarily compact) space $X$. The $d$--dimensional \v{C}ech homology group of $X$ is defined as the inverse limit of the inverse sequence of groups \[{\rm pro-}H_d(X) : \xymatrix{H_d(P_1) &   H_d(P_2) \ar[l]_-{(p_2)_*} & H_d(P_3) \ar[l]_-{(p_3)_*} & \ldots \ar[l]}\] where $(p_k)_*$ denotes the homomorphism induced by $p_k$ in $H_d$. The $d$--dimensional \v{C}ech cohomology group of $X$ is defined dually, using cohomology instead of homology. Choosing a basepoint $* \in X$, the $d$th shape group of $(X,*)$ is defined as the inverse limit of the inverse sequence of groups \[{\rm pro-}\pi_d(X,*) : \xymatrix{\pi_d(P_1,*) & \pi_d(P_2,*) \ar[l]_-{(p_2)_*} & \pi_d(P_3,*) \ar[l]_-{(p_3)_*} & \ldots \ar[l]}\] where again $(p_k)_*$ denotes the homomorphism induced by $p_k$ in $\pi_d$. 

It is true, but not obvious, that neither of these groups depends on the particular ANR expansion chosen to compute it. They are shape invariant: if two spaces $X$ and $Y$ have the same shape, then $\check{H}_d(X) = \check{H}_d(Y)$, $\check{H}^d(X) = \check{H}^d(Y)$ and $\check{\pi}_d(X,*) = \check{\pi}_d(Y,*)$. Finally, when $X$ is an ANR itself there are natural isomorphisms $\check{H}_d(X) = H_d(X)$, $\check{H}^d(X) = H^d(X)$ and $\check{\pi}_d(X,*) = \pi_d(X,*)$. 
\medskip

{\it\underline{The Mittag--Leffler property}}. An inverse sequence of groups \[\xymatrix{G_0 & G_1 \ar[l]_{\varphi_1} & G_2 \ar[l]_{\varphi_2} & \ldots \ar[l]_{\varphi_3}}\] has the \emph{Mittag--Leffler property} if for every $m$ there exists $N \geq m$ such that for every $n \geq N$ the equality ${\rm im}(\varphi_{m+1} \varphi_{m+2} \ldots \varphi_N) = {\rm im}(\varphi_{m+1} \varphi_{m+2} \ldots \varphi_n)$ holds true.

Given a pointed continuum $(X,*)$, we will be interested in whether the inverse sequence ${\rm pro-}\pi_d(X,*)$ has the Mittag--Leffler property. It turns out that this is independent of the particular ANR expansion chosen to define ${\rm pro-}\pi_d(X,*)$. As a particular case, when $X$ has the shape of a polyhedron then ${\rm pro-}\pi_d(X,*)$ has the Mittag--Leffler property.
\medskip

{\it\underline{Pointed $1$--movability}}. This condition plays an important role when studying attractors in $\mathbb{R}^3$. Instead of giving the original definition \cite[p. 198]{mardesic1} let us say that a continuum $X$ is \emph{pointed $1$--movable} if and only if ${\rm pro-}\pi_1(X,*)$ has the Mittag--Leffler property for some basepoint $*$ \cite[Theorem 4, p. 200]{mardesic1}, in which case the property holds for any basepoint \cite[Corollary 3, p. 212]{mardesic1}.

\bibliographystyle{plain}
\bibliography{biblio}

\end{document}